\documentclass[a4paper,12pt]{extarticle}
\usepackage{amsmath}
\usepackage{amssymb}
\usepackage{amsthm}

\usepackage{hyperref}
\hypersetup{
    colorlinks,
    citecolor=blue,
    filecolor=blue,
    linkcolor=blue,
    urlcolor=blue
}
\allowdisplaybreaks
\theoremstyle{thmstyleone}

\begin{document}
\newtheorem{theor}{Theorem}[section] 
\newtheorem{prop}[theor]{Proposition} 
\newtheorem{cor}[theor]{Corollary}
\newtheorem{lemma}[theor]{Lemma}
\newtheorem{sublem}[theor]{Sublemma}
\newtheorem{defin}[theor]{Definition}
\newtheorem{conj}[theor]{Conjecture}
\newtheorem{Rem}[theor]{Remark}

\gdef\beginProof{\par{\bf Proof: }}
\gdef\endProof{${\bf Q.E.D.}$\par}  
\gdef\ar#1{\widehat{#1}}
\gdef\pr{^{\prime}}
\gdef\prpr{^{\prime\prime}}
\gdef\mtr#1{\overline{#1}}
\def\bz{\mbox{\boldmath$\zeta$\unboldmath}}
\def\bzs{\mbox{\boldmath$\zeta'$\unboldmath}}
\def\ba{\mbox{\boldmath$a$\unboldmath}}
\def\bzo{\overline{\bz}}
\def\odd{{\rm odd}}
\gdef\ra{\rightarrow}
\gdef\Bbb{\bf }
\gdef\P1{{\Bbb P}^{1}_{D}}
\gdef\dbd{dd^{c}}
\gdef\a{\alpha}
\gdef\ca{ch_{g}(\alpha)}
\gdef\ttmk#1{\widetilde{\theta}^{k}(#1)}
\gdef\tdm#1{\theta^{k}({#1}^{\vee})^{-1}}
\gdef\td#1{\theta^{k}({#1}^{\vee})}
\gdef\rl{{\Lambda}}
\def\covol{{\rm covol}}
\def\lcovol{{\rm lcovol}}
\gdef\CT{CT}
\gdef\refeq#1{(\ref{#1})}
\gdef\mn{{\mu_{n}}}
\gdef\zn{{\Bbb Z}/(n)}
\gdef\umn{^{\mn}}
\gdef\lmn{_{\mn}}
\gdef\blb{{\big(}}
\gdef\brb{{\big)}}
\gdef\ttmk#1{\widetilde{\lambda}(#1)}
\gdef\chge#1{ch_{g}^{-1}(\lambda_{-1}({#1}^{\vee}))
ch_{g}(\lambda_{-1}({#1}^{\vee}_{g}))}
\gdef\mlcovol{{\rm mlc}}
\gdef\Hom{{\rm Hom}}
\def\Trs{{\rm Tr_s\,}}
\def\Tr{{\rm Tr\,}}
\def\End{{\rm End}}
\def\eq{equivariant }
\def\Td{{\rm Td}}
\def\ch{{\rm ch}}
\def\torus{{\cal T}}
\def\Proj{{\rm Proj}}
\def\bmn{{R_{n}}}
\def\uexp#1{{{\rm e}^{#1}}}
\def\Spec{{\rm Spec}\,}
\def\Qb{\mtr{\Bbb Q}}
\def\Zn{{\Bbb Z}/n}
\def\deg{{\rm deg}\,}
\def\mod{{\rm mod}}
\def\ac1{\ar{\rm c}_{1}}
\def\boxtimes{{\otimes_{\rm Ext}}}
\def\Qmm{{{\Bbb Q}(\mu_{m})}}
\def\NIm#1{{\rm Im}(#1)}
\def\NRe#1{{\rm Re}(#1)}
\def\rk{{\rm rk}\,}
\def\Uebungen{\subsection*{Aufgaben}}
\def\phi{\varphi}
\def\deg{{{\rm deg}\,}}
\def\Td{{\rm Td}}
\def\ch{{\rm ch}}
\def\Hom{{\rm Hom}}
\def\Bil{{\rm Bil}}
\def\End{{\rm End}}
\def\Sym{\Sym}
\gdef\endProof{\hfill$\square$\par\bigskip\noindent}
\def\Im{{\rm Im}\,}
\def\Re{{\rm Re}\,}
\def\im{{\rm im}\,}
\def\ker{{\rm ker}\,}
\def\Tr{{\rm Tr}\,}
\def\Eig{{\rm Eig}\,}
\def \x{\times}
\def \BN{{\bf N}}
\def \BZ{{\bf Z}}
\def \BQ{{\bf Q}}
\def \BR{{\bf R}}
\def \BC{{\bf C}}
\def \BH{{\bf H}}
\def\BP{{\bf P}}
\def\Sp{{\bf Sp}}
\def\GL{\mbox{\bf GL}}
\def\SL{\mbox{\bf SL}}
\def\SO{\mbox{\bf SO}}
\def\O{\mbox{\bf O}}
\def\SU{\mbox{\bf SU}}
\def\so{{\frak {so}}}
\def\Aut{\mbox{\rm Aut}}
\def\sign{\mbox{\rm sign\,}}
\def\ord{{\rm ord}}
\def \<{\langle}
\def \>{\rangle}
\def \Ad {{\rm Ad}}
\def \ad {{\rm ad}}
\def\vp{\varphi}
\def\vep{\varepsilon}
\def\theta{\vartheta}
\def\Ric{{\rm Ric}}
\def\vol{{\rm vol}}
\def\diam{{\rm diam}\,}
\def\xx{{\bf x}}
\def\Pf{{\rm Pf}}
\def\grad{{\rm grad}\,}
\def\E{E}
\def\F{{\bf F}} 
\def\eps{\varepsilon}
\def\div{{\rm div\,}}
\def\divv{{\nabla^*}}
\def\Sym{{\rm Sym}}
\def\Cas{{\rm Cas}}
\def\f{b}
\def\top{{\rm top}}
\def\J{J}
\def\mua{m^{\cal L}}
\def\XM{{X}}
\def\pp{{n}}


\def\bz{\mbox{\boldmath$\zeta$\unboldmath}}
\def\bzs{\mbox{\boldmath$\zeta'$\unboldmath}}
\def\odd{{\rm odd}}

\author{
Kai K\"ohler\footnote{Mathematisches Institut/
Universit\"atsstr. 1/
Geb\"aude 25.22/
D-40225 D\"usseldorf/
koehler@math.uni-duesseldorf.de}} 
\title{Analytic torsion forms for fibrations by projective curves}
\maketitle
\begin{abstract}
An explicit formula for analytic torsion forms for fibrations by projective curves is given. In particular one obtains a formula for direct images in Arakelov geometry in the corresponding setting. The main tool is a new description of Bismut's equivariant Bott-Chern current in the case of isolated fixed points.
\end{abstract}
\begin{center}
\end{center}
\thispagestyle{empty}
\setcounter{page}{1}

\section{Introduction}
The purpose of this paper is to make the computation of analytic torsion forms more accessible, exploring a method relying on a result by Bismut and Goette. We use this method to explicitly compute analytic torsion forms for fibrations by projective curves.

Analytic torsion forms have been constructed and investigated by Bismut and the author in \cite{BK} (1992) using heat kernels of certain differential operators. This definition followed previous constructions by Gillet-Soul\'e (\cite{GS1} 1986), Bismut-Gillet-Soul\'e (\cite{BGS2} 1988), Gillet-Soul\'e (\cite{GSZ} 1991). Further constructions and extensions  have been given by Faltings (\cite{Fa} 1992), Zha (\cite{Zha} 1998), Ma (\cite{Ma} 2000), Bismut (\cite{Bi10} 2013), Burgos Gil-Freixas i Montplet-Li\c tcanu (\cite{BuFrLi} 2014, including an axiomatic characterisation) and several other articles and books. Analytic torsion forms $T_\pi(\mtr E)$ are differential forms on the base $B$ associated to Hermitian holomorphic vector bundles $\mtr E$ over fibrations $\pi:M\to B$ of complex manifolds equipped with a certain K\"ahler structure. Their degree 0 part equals Ray-Singer's complex analytic torsion. Their main application is the construction of a direct image $\pi_!$ of Hermitian vector bundles in Gillet-Soul\'e's Arakelov $K$-theory of arithmetic schemes. This direct image is the sum of higher direct images on algebraic schemes plus $T_\pi$. Bismut's immersion formula for torsion forms (\cite{Bi11R}) enabled Gillet-R\"ossler-Soul\'e to prove a Grothendieck-Riemann-Roch theorem in Arakelov Geometry which relates $\pi_!$ to the direct image in Gillet-Soul\'e's Chow intersection theory of cycles and Green currents (\cite{GRS}, extending \cite{GS}). The torsion form also played a key role in Fu's and Zhang's proof of the birational invariance of BCOV torsion (\cite{Zha2},\cite{FZha}).

While there are many computations for the degree 0 part of analytic torsion forms, there are currently only few explicitly known values of analytic torsion forms in higher degree:
Torsion forms are known for vector bundles over torus bundles (\cite{K3}). Also Mourougane showed $T_\pi({\cal O})^{[2]}=0$ as the value in degree 2 of $T_\pi$ for the fibration  by Hirzebruch surfaces over ${\bf P}^1\BC$ (\cite{Mourougane}). Furthermore Bismut has shown that the equivariant torsion form of the $\BZ$-graded holomorphic de Rham complex of the fibers vanishes in cohomology (\cite{Bi13R}). Puchol proved a formula for asymptotic expansion of torsion forms for high powers of a line bundle (\cite{Puchol}), extended in degree 0 by Finski (\cite{Fi}).
In \cite[Remark 8.11]{Bi12R}, the explicit calculation of torsion forms for projective bundles was stated as an open problem with useful applications. We try to improve the situation by proving the following result in section \ref{FinalProof}:

\begin{theor}\label{mainTorsionformResult}
Let $\pi:E\to B$ be a holomorphic vector bundle of rank 2 over a complex manifold. Consider the formal power series $T_\ell\in \BR[[X]]$ given by
\begin{align*}
T_\ell(-t^2):=&
\sum_{m=1}^{|\ell+1|}\frac{\sin(2m-|\ell+1|)\frac{t}2}{\sin\frac{t}2}\log m
+\left(
\frac{\cos\frac{(\ell+1)t}2)}{t\sin\frac{t}2}
\right)^*
\\&
-\frac{\cos\frac{(\ell+1)t}2}{\sin\frac{t}2}\sum_{m\geq1\atop m{\rm\,odd}}\left(2\zeta'(-m)
+{\cal H}_m\zeta(-m)\right)\frac{(-1)^{\frac{m+1}2}t^m}{m!},
\end{align*}
where $(t^{2m})^*:=t^{2m}\cdot (2{\cal H}_{2m+1}-{\cal H}_m)$ with the harmonic numbers
\begin{equation}\label{harmonic}
{\cal H}_m=\sum_{j=1}^m\frac1j.
\end{equation}
The torsion form for $\mtr {{\cal O}(\ell)}$ on the ${\bf P}^1\BC$-bundle $\pi:{\bf P}E\to B$ is given by
$$
T_\pi(\mtr{{\cal O}(\ell)})=e^{-\frac\ell2c_1(E)}T_\ell(c_1(E)^2-4c_2(E))\in H^{2\bullet}(B,\BR).
$$
\end{theor}

Our method is as follows. It has been pointed out (after Atiyah-Singer (\cite[p. 133-134]{AS4})) by Bismut (\cite[p. 100-101]{Bi1}), Bismut-Gillet-Soul\'e (\cite{BGS2}) and Berline-Getzler-Vergne (\cite[ch. 10.7]{BeGeVe}) that Bismut super connections and torsion forms can be understood in a more accessible way if the fibration $\pi:M\to B$ and the associated objects are induced by a principle bundle $P\to B$ with compact structure group $G$. Bismut-Goette (\cite{BG}) use this to describe the torsion form in such a setting as a cohomology class which can be interpreted in terms of a $\frak g$-equivariant analytic torsion.

Bismut-Goette's main result relates this $\frak g$-equivariant analytic torsion to the $G$-equivariant analytic torsion introduced in \cite{K1} via Bismut's equivariant Bott-Chern current $S$. The construction of this Bott-Chern current was inspired by Mathai-Quillen's  very influential construction of a Gau\ss{} shape representative of the Thom class (\cite{MQ}). The crucial Gau\ss{} density in this construction makes explicit integration in our example difficult, and thus our strategy is to replace it by an indicator function closer to Thom's original construction (Th. \ref{SFormel1}). We do this in a general setting for isolated fixed points as this construction shall be applied to more general spaces in a forthcoming paper.

We also employ the formula for $S$ to demonstrate the usage of the residue formula in Arakelov theory (\cite[Th. 2.11]{KR2}) by applying it to ${\bf P}^1_\BZ$ in section \ref{TheheightofP1Z}. This residue formula (\`a la Bott) has never been applied before as the $S$-current makes explicit evaluations difficult. In Theorem \ref{g-eq.Torsion} we give an explicit formula for the $\frak g$-equivariant analytic torsion on ${\bf P}^1\BC$. In Theorem \ref{Gg-eq.Torsion}, we extend this to $(\frak g,G)$-equivariant analytic torsion providing the $G$-equivariant torsion form introduced in \cite{Ma}. In Remark \ref{vgl.bekannteTorsion} we verify that the degree 0 part of the formula in Theorem \ref{mainTorsionformResult} equals the known value of the Ray-Singer analytic torsion as given in \cite[Theorem 18]{K2}. Theorem \ref{CompareARR1} shows that the last summand in Theorem \ref{mainTorsionformResult} exactly cancels with another term in the arithmetic Grothendieck-Riemann-Roch Theorem from \cite{GRS}.

The author is indebted to the referee for a careful reading of this paper and for his comments.

\section{Equivariant characteristic classes}
Let $M$ be a complex manifold. Corresponding to the decomposition $TM\otimes\BC=TM^{1,0}\oplus TM^{0,1}$ define $U=U^{1,0}+U^{0,1}$ for $U\in TM\otimes\BC$. Let $\frak A^{p,q}(M)$ denote the vector space of forms of holomorphic degree $p$ and anti holomorphic degree $q$, and let $\frak A^{p,q}(M,E)$ denote the corresponding forms with coefficients in a holomorphic vector bundle $E$. Let $X\in\Gamma(M,TM)$ be a vector field such that its local flow acts holomorphically on $M$, i.e. $X^{1,0}$ is a holomorphic section of $T^{1,0}M$.

An $X$-equivariant holomorphic vector bundle $E$ equipped with an $X$-invariant Hermitian metric shall be denoted by $\mtr E$. Let $\nabla^E$ be the associated Chern connection with curvature $\Omega^E\in\frak A^{1,1}(M,\End\, E)$.

Following \cite[(2.7)]{BG} we denote by
$m^E(X):=\nabla^E_{X}-L^E_{X}\in\Gamma(M,\End\,E)$ the moment map as the skew adjoint endomorphism given by the 
difference between
the Lie derivative and the covariant derivative
on
$E$. In particular, for the flow $\Phi_t^X$ associated to $X$ and a zero $p$ of $X$, $m^{TX}(X)(p)=\left.\frac\partial{\partial t}\right\rvert_{t=0}\Phi^X_t(p)\in\End\, T_pM$.
Set as in \cite[(2.30), Def. 2.7]{BG} (compare \cite[ch. 7]{BeGeVe})
$$
\Td_{X}(\mtr {E}):=\Td\left(-\frac{\Omega^{E}}{2\pi i}+m^{E}(X)\right)
\in{\frak A}(M)
$$
and
$$
\ch_{X}(\mtr E):=\Tr \exp\left(-\frac{\Omega^E}{2\pi i}+m^E(X)\right)\in{\frak 
A}(M).
$$
The Chern class $c_{q,X}(\mtr E)$ for $0\leq q\leq \rk E$ is defined in \cite[Def. 2.5]{KR2} as the 
part of
total degree
$\deg_M+\deg_t=q$ of
$$
\det\left(\frac{-\Omega^E}{2\pi i}+tm^E(X)+{\rm id}\right)\in{\frak A}(M)
$$
at $t=1$, thus $c_{q,X}(\mtr E)=c_q(-\Omega^E/2\pi i+m^E(X))$. For $m^E$ invertible we set (\cite[(3.10)]{BG})
\begin{equation}
(c_{\top,X}^{-1})'(\mtr E):=\left.\frac{\partial}{\partial
b}\right\rvert_{b=0}c_{\rk E}\left(\frac{-\Omega^E}{2\pi i}+m^E(X)+b\, {\rm id}\right)^{-1}.
\end{equation}
The bundle $E$ splits
at every component of the fixed point set $M_X:=\{p\in M\mid X_p=0\}$ into a sum of holomorphic vector bundles
$\bigoplus E_\theta$ associated to eigenvalues $i\theta\in i{\bf R}$ of $m^E$. Let $I_{X}\in H^\bullet(M_X)$ denote the additive equivariant characteristic class which is given 
for a
line bundle $L$ as follows: If $X'$ acts at the fixed point $p$ by an angle
$\theta'\in{\bf R}^\x$ on
$L$, then
\begin{equation}
\left.I_{X'}(L)\right\rvert_{p}:=\sum_{k\in\BZ^\x}\frac{\log(1+\frac{\theta'}{2\pi 
k})}{c_1(L)+i\theta'+2 k\pi
i}.
\end{equation}

Next consider a holomorphic action $g$ on $M$. Assume that $E$ is $g$-invariant as a holomorphic Hermitian bundle and that $E$ is equipped with an equivariant structure $g^E$. The
Hermitian vector bundle $\mtr E$ splits on the fixed point submanifold $M_g$ into a direct sum
$\bigoplus_{\zeta\in S^1}\mtr E_\zeta$, where the equivariant structure $g^E$ of $E$
acts on
$\mtr E_\zeta$ as $\zeta$. Then the \emph{$g$-equivariant Chern character form} is defined as
\begin{eqnarray*}
\ch_g(\mtr E)&:=&\sum_\zeta \zeta\ch(\mtr E_\zeta)\\ &=&\Tr g^E+\sum_\zeta \zeta
c_1(\mtr E_\zeta)+\sum_\zeta \zeta \left(\frac{1}{2}c_1^2(\mtr E_\zeta) -c_2(\mtr
E_\zeta)\right)+\dots\in{\frak A}(M_g)
.
\end{eqnarray*}
Thus, $\widetilde\ch_g(\mtr E)=\sum_\zeta
\zeta\widetilde\ch(\mtr E_\zeta)$.  
With the $g$-invariant subbundle $E_1\to M_g$, the Todd form of a $g$-equivariant vector bundle is defined as
$$
\Td_g(\mtr E):=\frac{c_{{\rm rk}\,E_1}(\mtr E_1)}
{\ch_g(\sum_{j=0}^{{\rm rk}\,E} (-1)^j \Lambda^j \mtr E^*)}.
$$
\section{Analytic torsion forms}
In this section we describe the definition of equivariant Ray-Singer analytic torsions and analytic torsion forms. We simplify the more general setting in \cite{BK} a bit for the sake of exposition, as we shall use the torsion forms in this article only in a very restricted setting.

Let $M$ be a compact K\"ahler manifold of complex dimension $\pp$ with K\"ahler form $\omega^{TM}\in\frak A^{1,1}(M)$. We choose the K\"ahler form such that it verifies the condition $\omega^{TM}(U,V)=g^{TM}(JU,V)$ (note that \cite{BK}, \cite[p. 1302]{BG} use $-\omega^{TM}$ instead as the K\"ahler form). Thus $\omega^{TM}_p=\sum_{j=0}^{\pp}dx_{2j-1}\wedge dx_{2j}$ in geodesic coordinates at an origin $p$. For $U\in T^{0,1}M$, $q\in\BN_0$, let $\iota_U:\Lambda^{q} T^{*0,1}M\to \Lambda^{q-1} T^{*0,1}M$ denote the interior product antiderivation.
Define fibrewise an action of the Clifford algebra assciated to $(TM,g^{TM})$ on $\Lambda^\bullet T^{*0,1}M\otimes E$ by
$$
c(U):=\sqrt{2}(g^{TM}(\cdot, U^{1,0})\wedge)-\sqrt{2}\iota_{U^{0,1}}\qquad(U\in TM).
$$
Let $N_\infty:\Lambda^\bullet T^*M\otimes E\to\BN_0$ map each component to its differential form degree. Consider a holomorphic isometric action of a Lie group $G$ on $M$. Consider $g\in G$ and a vector field $X$ induced by an element of the Lie algebra $\frak z_{G}(g)\subset\frak g$ of the centralizer of $g$. As above let $\bar E\to M$ be an $X$-equivariant Hermitian holomorphic vector bundle.
Assume that the action of $X$ on $(M,\omega)$ is Hamiltonian, i.e. there exists a function $\mu\in C^\infty(M,\BR)$ such that $d\mu=\iota_X\omega$. This implies 
$X.\mu=0$, and for $M$ connected $\mu$ is uniquely determined up to constant. If $L\to M$ is an $X$-invariant polarized variety and $\omega:=i\Omega^L$, one can choose $\mu:=-i\cdot m^L$. With the Dolbeault operator associated to $E$, set as in \cite[(2.40)]{BG}
$$
C_{X,t}^M:=\sqrt{t}(\bar\partial^M+\bar\partial^{M*})
+\frac1{2\sqrt{2t}}c(X)
$$
acting on $\frak A^{0,\bullet}(M,E)$.
\begin{defin}\rm (\cite[p. 1319]{BG}) 
For $s\in\BC$, $\Re s\in]0,\frac12[$ and $|X|$ sufficiently small, the zeta function
\begin{eqnarray*}
Z(s)&:=&\frac{-1}{\Gamma(s)}\int_0^\infty t^{s-1}
\Bigg(\Tr_s\left(
N_\infty-\frac{i \mu}{t}
\right)g\\
&&\cdot\exp(-L_X-(C_{X,t}^M)^2)
-\Tr_s^{H^\bullet(M,E)}(N_\infty ge^X)
\Bigg)\,dt.
\end{eqnarray*}
is well-defined and $Z$ has a holomorphic continuation to $s=0$. The \emph{$(g,X)$-equivariant complex Ray-Singer torsion} is defined as $T_{g,X}(M,\mtr E):=\left.\frac\partial{\partial s}\right\rvert_{s=0}Z(s)$.
\end{defin}
The \emph{$g$-equivariant torsion $T_g(M,\mtr E)$} was defined in \cite{K1}, and Bismut-Goette's Definition extends this such that $T_g(M,\mtr E)=T_{g,0}(M,\mtr E)$.

\begin{defin}\rm(\cite[Def. 1.4]{BGS2})
Let $\pi:M\stackrel{Z}\to B$ be a proper holomorphic submersion of complex manifolds $M$, $B$. Let $TZ$ and $TZ^\perp$ denote the vertical tangent bundle and the horizontal distribution othogonal to it, respectively. Suppose that there exists a closed 2-form $\omega\in\frak A^{1,1}(M)$ such that $g^{TZ}:=\left.\omega\right\rvert_{TZ^{\otimes2}}(\cdot,J\cdot)$ is Hermitian. Then $(\pi,g^{TZ},TZ^\perp)$ is called a \emph{K\"ahler fibration}.\end{defin}

For $b\in B$ set $Z_b:=\pi^{-1}(\{b\})$. For any $U\in T_bB$ we denote by $U^H\in (TZ_b)^\perp\subset TM$ the horizontal lift to the orthogonal complement of the vertical tangent space. Let $g^{TB}$ be a metric on $B$, inducing $\nabla^{(TZ^\perp)}$. Set $\nabla^{TM}:=\nabla^{(TZ^\perp)}\oplus\nabla^{TZ}$ with torsion ${\cal T}\in{\frak A}^{1,1}(M,TZ)$. Set $\omega^H\in{\frak A}^{1,1}(B)\otimes C^\infty(M)$, $\omega^H(U,U'):=\omega(U^H,{U'}^H)$, ${\cal T}^H(U,U'):={\cal T}(U^H,{U'}^H)$. Let $\bar \E\to M$ be an Hermitian holomorphic vector bundle. 
Let $F\to B$ denote the $\infty$-dimensional vector bundle with fibre
$$
F_b:=\Gamma^\infty(Z_b,\Lambda^\bullet T^{*0,1}Z_b\otimes \left.\E\right\rvert_{Z_b})
$$
and connection $\nabla^F_Us:=\nabla^{\Lambda^\bullet T^{*0,1}Z\otimes \E}_{U^H}s$.

\begin{defin}\rm(\cite[Def. 1.8]{BK})
For $t\in\BR^+$, the \emph{number operator} $N_t\in\Gamma(B,\Lambda^\bullet T^*B\otimes\End\, F)$ is given by
$$
N_t:=N_\infty-\frac{i\omega^H}{t}.
$$

The \emph{Bismut super connection} on $F$ is defined using the Clifford operation $c$ of the $TZ$ component of $\cal T$ on $\Lambda^\bullet T^{*0,1}Z$ as
$$
B_t:=\nabla^F+C^Z_{-{\cal T}^H,t}.
$$
\end{defin}
The operator $B_t$ is formed as an adiabatic limit of the Dirac operator on $M$. As differential forms on the manifold $B$, the summands have the degrees $1,0,2$. 
\begin{defin}\rm(\cite[Def. 3.8]{BK})
Set $\tilde{\frak A}(B):=\bigoplus_p{\frak A^{p,p}(B)}/(\im\partial+\im\bar\partial)$. Assume $H^\bullet(Z_\cdot,\left.E\right\rvert_{Z_\cdot})\to B$ to be vector bundles. For $|\Re s|<\frac12$ set (regularized as in \cite[(3.10)]{BK})
\begin{eqnarray*}
Z(s)&:=&\frac{-1}{\Gamma(s)}\int_0^\infty t^{s-1}(2\pi i)^{-N_\infty/2}\\
&&\cdot\left(\Tr_s N_t e^{-B_t^2}-\Tr_sN_\infty e^{-\Omega^{H^\bullet(Z_\cdot,{\cal E})}}\right)\,dt\in\bigoplus_p{\frak A^{p,p}(B)}.
\end{eqnarray*}
The \emph{analytic torsion form} associated to the K\"ahler fibration $\pi$ and $\bar\E$ is defined as
$
T_\pi(\bar\E):=Z'(0)\in\tilde{\frak A}(B).
$
\end{defin}

In degree 0 one gets $\left.T_\pi(\bar\E)^{[0]}\right\rvert_{b}=T_{\rm id}(Z_b,\left.\bar\E\right\rvert_{Z_b})$ with the Ray-Singer torsion on the right hand side.

{\bf Examples} 1) (\cite[Th. 4.1]{K3})
Let $\bar E\to B$ be an Hermitian holomorphic vector bundle of rank $k$, $\Lambda\subset E$ a $\BZ^{2k}$-bundle with holomorphic local sections.  Then $\pi:M:=E/\Lambda\to B$ is a torus bundle. For $\Re s<0$ set
$$
Z(s):=\frac{\Gamma(2k-1-s)}{(2\pi)^k(k-1)!\Gamma(s)}\sum_{\lambda\in\Lambda\atop\lambda\neq0}
\frac{(\partial\bar\partial\|\lambda\|^2)^{\wedge(k-1)}}
{(\|\lambda\|^2)^{2k-s-1}}
\in{\frak A}^{k-1,k-1}(B).
$$
Then $T_\pi({\cal O})=\frac{Z'(0)}{\Td(\bar E)}\in\tilde{\frak A}(B)$. A K\"ahler fibration condition is not necessary. 

2) Mourougane considered Hirzebruch surfaces
$$\pi:F_k={\bf P}({\cal O}(-k)\oplus{\cal O})\to{\bf P}^1\BC$$
and obtained 0 as the part of $T_\pi({\cal O})$ in degree 2 (\cite[p. 239]{Mourougane}).

\section{\texorpdfstring{Bismut's equivariant Bott-Chern current}{Definition of S}}
The equivariant Bott-Chern current has been introduced and investigated by Bismut in \cite{Bi8}, \cite{Bi9}. In this section we briefly cite some of its properties following the presentation in \cite{BG}. 
In the special case of isolated fixed points we shall necessarily obtain these results independently in the next section to obtain our expression for $S_X$.
Let $M$ be a compact K\"ahler manifold acted upon by a holomorphic Killing field $X$. Denote (\cite[p. 1312, Def. 2.6]{BG})
$$
d_X:=d-2\pi i\iota_{\XM},\quad
\partial_X:=\partial-2\pi i\iota_{X^{0,1}},\quad
\bar\partial_X:=\bar\partial-2\pi i\iota_{X^{1,0}}.
$$
The holomorphy of $X$ implies $\partial_X^2=0$, $\bar\partial_X^2=0$. Notice that $d_X^2=-2\pi i L_{\XM}$.

Let $N^*_{\BR}$ denote the dual of the real normal bundle of the embedding $M_X\hookrightarrow M$. Let (\cite[Def. 3.5]{BG}) $P_{X,M_X}^M$ be the set of currents $\alpha$ on $M$ with wave front set included in $N^*_{\BR}$, such that $\alpha$ is a sum of currents of type $(p,p)$ and $L_{\XM}\alpha=0$. Let $P_{X,M_X}^{M,0}\subset P_{X,M_X}^M$ denote the subset consisting of those $\alpha=\partial_X\beta+\bar\partial_X\beta'$, where $\beta,\beta'$ are $\XM$-invariant currents whose wave front set is included in $N^*_{\BR}$. We shall use the notation $X^\flat\in \Gamma(M,T^*M)$ for the metric dual of a vector field $X\in \Gamma(M,TM)$.
Set (\cite[p. 1322, Def. 3.3]{BG})
\begin{eqnarray}\nonumber\label{dt}
d_t&:=&\frac{\omega^{TM}}{2\pi t}\exp\left(\frac{\bar\partial_X\partial_X}{2\pi i t}\frac{-\omega^{TM}}{2\pi}\right)
\stackrel{\rm \cite[Prop.\,3.2]{BG}}=\frac{\omega^{TM}}{2\pi t}\exp\left(d_X\frac{\XM^\flat}{4\pi i t}\right).
\\&=&\frac{\omega^{TM}}{2\pi t}\exp\left(\frac1t(\frac{d(\XM^\flat)}{4\pi i}-\frac12\|\XM\|^2)\right).
\end{eqnarray}
Then (\cite[(3.9)]{BG}) there is a current $\rho_1\in P^M_{X,M_X}$ such that for $t\to0^+$ and any $\eta\in\frak A(M)$,
\begin{eqnarray}\label{3.9BG}
\int_M\eta\cdot d_t&=&\frac1t\int_{M_X}\eta\cdot\frac{\omega^{TM}/2\pi}{c_{\top,X}(N_{M_X/M})}
+\int_M\eta\cdot\frac{\omega^{TM}}{2\pi}\rho_1+O(t).
\end{eqnarray}
By equation (\ref{dt}), $F^2_\eta(s):=\frac1{\Gamma(s)}\int_1^\infty(\int_M\eta d_t)t^{s-1}dt$ is well-defined and holomorphic for ${\rm Re\,}s<1$ (\cite[(3.13)]{BG}). 
Similarly equation (\ref{3.9BG}) shows that
$$
F^1_\eta(s):=\frac1{\Gamma(s)}\int_0^1(\int_M\eta d_t)t^{s-1}dt,\qquad
$$
is well-defined for ${\rm Re\,}s>1$ and that it has a holomorphic continuation to $s=0$.
Thus one can set (\cite[p. 1324, Def. 3.7]{BG}) $\int_M\eta S_X(M,-\omega^{TM}):=\left.\frac\partial{\partial s}\right\rvert_{s=0}(F^1_\eta+F^2_\eta)$.
By \cite[Th. 3.9]{BG} $S_X(M,-\omega^{TM})\in P^M_{X,M_X}$. By \cite[Prop. 3.8]{BG} or \cite[Prop. 2.11]{Bi9} one gets
\begin{eqnarray*}
\lefteqn{\int_M\eta S_X(M,-\omega^{TM})}\\
&=&
\int_0^1\int_M\eta\left(
d_t-\frac{\omega^{TM}}{2\pi t c_{\top,X}(N_{M_X/M},g^{TM})}\delta_{M_X}-\frac{\omega^{TM}}{2\pi}\rho_1
\right)\frac{dt}t
\\&&+\int_1^\infty(\int_M\eta d_t)\frac{dt}t-\int_{M_X}\eta\frac{\omega^{TM}}{2\pi c_{\top,X}(N_{M_X/M},g^{TM})}
-\Gamma'(1)\int_M\eta\frac{\omega^{TM}}{2\pi}\rho_1.
\end{eqnarray*}

According to \cite[p. 1323, Th. 3.6]{BG} (or \cite[(40)-(49)]{Bi8},\cite[Th. 2.7]{Bi9})
\begin{equation}\label{3.6BG}
\frac{\omega^{TM}}{2\pi}\rho_1=(c_{\top,X}^{-1})'(N_{M_X/M},g^{TM})\delta_{M_X}
\end{equation}
up to currents in $P^{M,0}_{X,M_X}$. Thus when $\eta$ is a $d_X$-closed form,
\begin{eqnarray}\nonumber
\lefteqn{\int_M\eta S_X(M,-\omega^{TM})}
\\&=&
\int_0^1\int_M\eta\left(
d_t-\frac{\omega^{TM}}{2\pi t c_{\top,X}(N_{M_X/M},g^{TM})}\delta_{M_X}-(c_{\top,X}^{-1})'(N_{M_X/M},g^{TM})\delta_{M_X}
\right)\frac{dt}t
\nonumber
\\&&+\int_1^\infty(\int_M\eta d_t)\frac{dt}t-\int_{M_X}\eta\frac{\omega^{TM}}{2\pi c_{\top,X}(N_{M_X/M},g^{TM})}
\nonumber
\\&&-\Gamma'(1)\int_{M_X}\eta(c_{\top,X}^{-1})'(N_{M_X/M},g^{TM}).
\end{eqnarray}
By replacing the first integral by $\lim_{a\to0}\int_a^1$ and using $\int_a^1\frac{dt}{t^2}+1=\frac1a$, $\int_a^1\frac{dt}t=-\log a$, $\int_a^\infty e^{c/t}\frac{dt}{t^2}\stackrel{\Re c<0}=\frac{e^{c/a}-1}c$ (and the variant of the last equation for $\int_a^\infty e^{c/t}\frac{dt}{t^{2+m}}$, $m\geq0$) as in \cite[p. 96]{KR2} this becomes
\begin{multline}\label{limaSX}
\int_M\eta S_X(M,-\omega^{TM})=\lim_{a\to0^+}
\Bigg(\int_M\eta \frac{-\omega^{TM}}{2\pi}\cdot\frac{1-\exp(d_X\frac{\XM^\flat}{4\pi i a})}{d_X\frac{\XM^\flat}{4\pi i}}
\\+\int_{M_X}\eta\left(\frac{-\omega^{TM}}{2\pi a c_{\top,X}(N_{M_X/M},g^{TM})}
+(\log a-\Gamma'(1))(c_{\top,X}^{-1})'(N_{M_X/M},g^{TM})\right)\Bigg).
\end{multline}

We shall need in the case of isolated fixed points a sharper version of equations (\ref{3.9BG}), (\ref{3.6BG}): $\int_M\eta\cdot d_t=\int_{M_X}\eta (c_{\top,X}^{-1})'(TM)+O(t)$ for any smooth form $\eta$ (Lemma \ref{BGVerify}) and not only for $d_X$-closed forms.

The dependence of $S_X(M,-\omega^{TM})$ on $\omega^{TM}$ is analysed in \cite[Th. 3.10]{BG}:
\begin{theor}\label{BG3.10}\cite[Th. 3.10]{BG} For K\"ahler forms $\omega^{'TM},\omega^{TM}$ on $M$ and the induced metrics $g^{'TM},g^{TM}$ on $N_{M_X/M}$,
$$
S_X(M,-\omega^{'TM})-S_X(M,-\omega^{TM})=-\widetilde{c^{-1}_{\top,X}}(N_{M_X/M},g^{'TM},g^{TM})\cdot\delta_{M_X}
$$
in $P^M_{X,M_X}/P^{M,0}_{X,M_X}$.
\end{theor}
An important special case arises when rescaling $\omega^{'TM}=c^2\omega^{TM}$: We shall use the notation $\alpha^{[q]}$ for the part in degree $q$ of a differential form $\alpha$. Because of $\widetilde{\ch^{[q]}}(E,c^2h^E,h^E)=\left(\widetilde\ch(\BC,c^2|\cdot|^2,|\cdot|^2)\ch(E,h^E)\right)^{[q-2]}$ by \cite[(1.3.5.2)]{GS} and $\widetilde\ch(\BC,c^2|\cdot|^2,|\cdot|^2)=-\log c^2$ (\cite[(1.2.5.1)]{GS}),
one finds
\begin{eqnarray*}
\widetilde{\ch^{[q]}}(E,c^2h^E,h^E)
&=&-\log c^2\cdot \ch(E,h^E)^{[q-2]}
\\&=&-\log c^2\cdot\left.\frac\partial{\partial t}\right\rvert_{t=0}\ch^{[q]}\left(-\frac{\Omega^E}{2\pi i}+t{\rm id}_E\right).
\end{eqnarray*}
Thus this relation holds when replacing $\ch^{[q]}$ by any other polynomial in the Chern classes, in particular
$$
-\widetilde{c^{-1}_{\top,X}}(N_{M_X/M},c^2g^{TM},g^{TM})=\log c^2\cdot(c^{-1}_{\top,X})'(N_{M_X/M},g^{TM}).
$$
This way Theorem \ref{BG3.10} implies a useful formula for the dependence of $S_{tX}$ on $t$, which we shall verify independently for isolated fixed points on the level of currents in Corollary \ref{ScalingSXeigene}:
\begin{cor}\label{ScalingSX}
Let $N_\infty:\frak A^\bullet(M)\to \BN_0$ denote the number operator. Then
$$
c^{N_\infty/2+1}S_{cX}(M,-\omega^{TM})-S_X(M,-\omega^{TM})=\log c^2\cdot(c^{-1}_{\top,X})'(N_{M_X/M},g^{TM})\cdot\delta_{M_X}
$$
in $P^M_{X,M_X}/P^{M,0}_{X,M_X}$.
\end{cor}
Note that $P^M_{X,M_X},P^{M,0}_{X,M_X}$ are invariant under rescaling of $X$.
\begin{proof}
When replacing $\omega^{TM}$ by $\omega':=b\omega^{TM}$ with $b\in\BR^+$, the corresponding form $d_t'$ is given by $d_t'=d_{t/b}$. On the other hand when replacing $X$ by $\tilde X:=cX$ with $c\neq0$, the associated form $\tilde d_t$ equals
\begin{eqnarray*}
\tilde d_t&=&\frac{\omega^{TM}}{2\pi t}\exp\left(\frac1t(\frac{cd\XM^\flat}{4\pi i}-\frac{c^2}2\|\XM\|^2)\right)
\\&=&c^{-1}\frac{\omega^{TM}/c}{2\pi t/c^2}\exp\left(\frac1{t/c^2}(\frac{d\XM^\flat/c}{4\pi i}-\frac{1}2\|\XM\|^2)\right)
\\&=&c^{-N_\infty/2-1}d_{t/c^2}.
\end{eqnarray*}
Thus $S_{cX}(M,-\omega^{TM})=c^{-N_\infty/2-1}S_X(M,-c^2\omega^{TM})$ and the result follows from Theorem \ref{BG3.10},
\[
c^{N_\infty/2+1}S_{cX}(M,-\omega^{TM})-S_X(M,-\omega^{TM})=-\widetilde{c^{-1}_{\top,X}}(N_{M_X/M},c^2g^{TM},g^{TM})\cdot\delta_{M_X}
\]
\end{proof}
\begin{Rem}\rm
If $M$ is compact, its Lie group of isometries is compact and thus the closure of the subgroup generated by a Killing field $X$ is a compact torus $T$. Thus in this case $\eta$ can be made $X$-invariant by taking the mean value $\left.\tilde\eta\right\rvert_{q}:=\frac1{\vol T}\int_T \left.(y^*\eta)\right\rvert_{q}\,d\vol_y$. As the equivariant Bott-Chern current is $X$-invariant, one can make the substitution $\int_M\eta S_X(M,-\omega^{TM})=\int_M\tilde\eta S_X(M,-\omega^{TM})$ and thus always assume that $L_X\eta=0$. The condition $d_X\eta=0$ then is a cohomological condition.
\end{Rem}

\section{\texorpdfstring{A formula for the equivariant Bott-Chern current}{Verification of the formula for the S-class}}
As in the last section $M$ shall be a compact K\"ahler manifold and $X$ a holomorphic Killing field.
For all of the results in this section, $M$ can as well be any compact subset of a (possibly non-compact) K\"ahler manifold $\tilde M$ and $X$ can be a holomorphic Killing field on $\tilde M$ without any zeros on $\partial M$, when $TM,d\vol_M,\frak A(M)$ are replaced by $T\tilde M,d\vol_{\tilde M},\frak A(\tilde M)$.

We assume that the zero set $M_X=:\{p_\ell\in M\mid \ell\in \J\}$ of $X$ has dimension 0. Set
$$\nu:=\frac{dX^\flat}{4\pi i}$$
(while for large parts of the results it could be any differential form of degree 2).

Choose $R$ small enough such that the connected component $B_\ell$ of $p_\ell$ in $\{q\in M\mid \|X_q\|\leq R\}$ can be covered by a chart and such that $B_\ell$ does not contain another zero. For a fixed $\ell\in\J$ we shall denote the corresponding coordinates by $x$, chosen such that $x=0$ at $p_\ell$. Let $\|\cdot\|_{\rm eucl}$, $d\lambda$ denote the euclidean metric and the Lebesgue measure in this chart.

Part (2) of the following Proposition gives a first estimate for the right hand side in equation (\ref{limaSX}).
\begin{prop}\label{estimate1}\ 
\begin{enumerate}
\item 
There exists $c_0\in\BR^+$ such that for all $s\in\BN_0$, $s<\pp$, $a>0$, $m\in\BR$ 
\begin{eqnarray*}
\int_{B_\ell}a^{-m}\|X\|^{-2s}e^{-\frac1{2a}\|X\|^2}d\vol_M&<&c_0a^{\pp-m-s}.
\end{eqnarray*}
\item
For $a\to0^+$ and any $\tilde\eta\in\frak A(M)$ with $\deg\tilde\eta\geq2$,
\begin{multline*}
\int_{B_\ell}\tilde\eta\frac{1-e^{\frac{\nu-\frac12\|X\|^2}a}}{\nu-\frac12\|X\|^2}
-\int_{B_\ell}\tilde\eta\frac{\nu^{\pp-1}\left(e^{-\frac1{2a}\|X\|^2}-1\right)}{(\frac12\|X\|^2)^{\pp}}
\\=
-\sum_{s=1}^{\pp-1} \int_{B_\ell}\tilde\eta\frac{\nu^{s-1}}{(\frac12\|X\|^2)^{s}}+
 \sum_{m=1}^{\pp-1}\int_{B_\ell}\tilde\eta e^{-\frac1{2a}\|X\|^2}
\cdot\nu^{\pp-1}\frac{(\frac12\|X\|^2)^{-\pp+m}}{m!a^{m}}
+O(a).
\end{multline*}
\end{enumerate}
\end{prop}
In part (2) the summands on the right hand side converge for $a\to0^+$. In general the summands on the left hand side do not converge.
\begin{proof}
1) Choose $c,C>0$ such that $c\|x\|_{\rm eucl}<\left.\|X\|\right\rvert_{x}$ and $d\vol_{M|x}=f(x)\cdot d\lambda$ with $|f|<C$ and, as defined above, $d\lambda$ being the Lebesgue measure. Then for $0\leq s<\pp$,
\begin{multline}\label{simplePower}
\int_{B_\ell}\|X\|^{-2s}d\vol_M<\int_{B_\ell}(c\|x\|_{\rm eucl})^{-2s}C\,d\lambda
<\int_{c\|x\|_{\rm eucl}<R}(c\|x\|_{\rm eucl})^{-2s}C\,d\lambda
\\=\frac{C}{c^{2s}}\vol(S^{2\pp-1})\cdot\int_0^Rr^{2\pp-1-2s}dr=\frac{C}{c^{2s}}\vol(S^{2\pp-1})\cdot \frac{R^{2\pp-2s}}{2\pp-2s}.
\end{multline}
Similarly,
\begin{multline*}
\int_{B_\ell}\|X\|^{-2s}e^{-\frac1{2a}\|X\|^2}d\vol_M<\int_{B_\ell}(c\|x\|_{\rm eucl})^{-2s}e^{-\frac{c^2}{2a}\|x\|_{\rm eucl}^2}C\,d\lambda
\\<\int_{\BR^{2\pp}}(c\|x\|_{\rm eucl})^{-2s}e^{-\frac{c^2}{2a}\|x\|_{\rm eucl}^2}C\,d\lambda
=\frac{Ca^{\pp-s}}{c^{2\pp}}
\int_{\BR^{2\pp}}\frac{1}{\|x\|_{\rm eucl}^{2s}}e^{-\frac{\|x\|_{\rm eucl}^2}{2}}\,d\lambda.
\end{multline*}
The integral on the right hand side exists by inequality (\ref{simplePower}) (or, of course, classically). Thus there is a constant $c_0$ depending only on $c,C,n$ and $s$ such that
\begin{eqnarray*}
\int_{B_\ell}\|X\|^{-2s}e^{-\frac1{2a}\|X\|^2}d\vol_M&<&c_0a^{\pp-s}.
\end{eqnarray*}
As $s$ varies over a finite range, $c_0$ can be chosen independently of $s$.

2) For $\frac12\|X\|^2\neq0$ Taylor expansion shows
\begin{eqnarray*}
\frac1{\nu-\frac12\|X\|^2}=\frac1{-\frac12\|X\|^2}\sum_{s=1}^\infty(\frac{\nu}{\frac12\|X\|^2})^{s-1}=-\sum_{s=1}^\infty\frac{\nu^{s-1}}{(\frac12\|X\|^2)^{s}}
\end{eqnarray*}
where the sum is finite, as $\nu$ has vanishing degree 0 part. 
Additionally expanding $\exp$ shows
\begin{multline*}
\int_{B_\ell}\tilde\eta\frac{1-e^{\frac{\nu-\frac12\|X\|^2}a}}{\nu-\frac12\|X\|^2}
=\int_{B_\ell}\tilde\eta
\sum_{s=1}^\infty\frac{-\nu^{s-1}}{(\frac12\|X\|^2)^{s}}\left(1-e^{-\frac{\|X\|^2}{2a}}
\sum_{m=0}^\infty\frac{(\frac{\nu}a)^m}{m!}\right)
\\=\int_{B_\ell}\tilde\eta\left(
\sum_{s=1}^\infty\frac{-\nu^{s-1}}{(\frac12\|X\|^2)^{s}}
+e^{\frac{-\|X\|^2}{2a}}\cdot\sum_{s=1}^\infty\sum_{m=0}^\infty\frac{\nu^{m+s-1}(\frac12\|X\|^2)^{-s}}{m!a^{m}}
\right)
\end{multline*}
where summands with $s+m>\pp$ are vanishing. For $s<\pp$, the integral over the first summand exists by inequality (\ref{simplePower}). For $s<\pp$ and $a\to0^+$, by part (1) the integral over the second summand converges if $s+m\leq\pp$ and it equals $O(a)$ for $s+m<\pp$.
\end{proof}
Choose an oriented orthonormal base of $T_{p_\ell}M$ with
$$\left.(\nabla X)\right\rvert_{p_\ell}={\tiny
\left(
\begin{array}{ccc}
0  & -\theta_1  &   \\
\theta_1  & 0  &   \\
  &   &  \ddots 
\end{array}
\right)}=:A.$$
In the corresponding geodesic coordinates,
\begin{eqnarray*}
X_{x}&=&Ax=(-\theta_1x_2,\theta_1x_1,\dots)^t,
\\
X^\flat&=&\theta_1(-x_2\,dx_1+x_1\,dx_2)+\dots+o(\|{x}\|)
\\
\mbox{and }\left.dX^\flat\right\rvert_{p_\ell}&=&2\sum_{j=1}^{\pp}\theta_jdx_{2j-1}\wedge dx_{2j}.
\end{eqnarray*}
\begin{prop}\label{tildeEtaAtpl} Define $a_\ell$ via $\left.(\tilde\eta\wedge\nu^{\pp-1})\right\rvert_{p_\ell}=:a_\ell\,d\vol_M$. If $\tilde\eta=\eta\wedge\frac{\omega^{TM}}{2\pi}$, then
$$
a_\ell
=\frac{\eta^{[0]}_{p_\ell}}{2^{\pp-1}\vol (S^{2\pp-1})}(c_{\top,X}^{-1})'(TM)\prod_j\theta_j^2.
$$
\end{prop}
\begin{proof}
With $\left.\omega^{TM}\right\rvert_{p_\ell}=\sum_{j=1}^{\pp}dx_{2j-1}\wedge dx_{2j}$ we get
\begin{align*}
\left.\frac{\omega^{TM}}{2\pi}\wedge\left(\frac{dX^\flat}{4\pi i}\right)^{\pp-1}\right\rvert_{p_\ell}
&=\frac{2^{\pp-1}\left(\pp-1\right)!}{2\pi(4\pi i)^{\pp-1}}\sum_{j=1}^{\pp}\frac1{\theta_j}\prod_{j=1}^{\pp}\theta_j\cdot d\lambda
\\&=\frac1{(2 i)^{\pp-1}\vol (S^{2\pp-1})}\sum_{j=1}^{\pp}\frac1{\theta_j}\prod_{j=1}^{\pp}\theta_j\cdot d\lambda
\\&=\frac1{2^{\pp-1}\vol (S^{2\pp-1})}(c_{\top,X}^{-1})'(TM)\prod_j\theta_j^2\cdot d\lambda.
\end{align*}
\end{proof}
The next Proposition further simplifies the terms in Proposition \ref{estimate1}(2): (1) computes a limit for the second summand on the right hand side, and (2) simplifies the second summand on the left hand side.
\begin{prop}\label{limitsFora}
As $a\to0^+$ the following estimates hold:
\begin{enumerate}
\item For $\pp=s+m$ and $0\leq s<\pp$
\begin{eqnarray*}
\int_{B_\ell}\frac{\tilde\eta\wedge\nu^{\pp-1}}{a^mm!(\frac12\|X\|^2)^{s}}e^{-\frac{1}{2a}\|X\|^{2}}
=
a_\ell\frac{\vol(S^{2\pp-1})}{\prod^{\pp}\theta_j^2}\frac{2^{\pp-1}}{m}+O(a).
\end{eqnarray*}
\item For $\pp=s$ and $B_\ell'(a):=\{x\in B_\ell\mid \|X_x\|^2<2a\}$, 
\begin{eqnarray*}
\lefteqn{\int_{B_\ell}
\frac{\tilde\eta\wedge\nu^{\pp-1}}{(\frac12\|X\|^2)^{\pp}}\left(
e^{-\frac1{2a}\|X\|^2}-1
\right)
+\int_{B_\ell\setminus B_\ell'(a)}\tilde\eta
\frac{\nu^{\pp-1}}{(\frac12\|X\|^2)^{\pp}}
}
\\&=&\int_{B_\ell'(a)}\tilde\eta
\frac{\nu^{\pp-1}}{(\frac12\|X\|^2)^{\pp}}\left(
e^{-\frac1{2a}\|X\|^2}-1
\right)
+\int_{B_\ell\setminus B_\ell'(a)}\tilde\eta
\frac{\nu^{\pp-1}}{(\frac12\|X\|^2)^{\pp}}
e^{\frac{-\frac12\|X\|^2}{a}}
\\&=&
a_\ell\frac{\vol(S^{2\pp-1})}{\prod^{\pp}\theta_j^2}2^{\pp-1}\Gamma'(1)+O(a).
\end{eqnarray*}
\end{enumerate}
\end{prop}
We shall use the notation $B^{2\pp}_R(x)\subset \BR^{2\pp}$ for the euclidean ball of radius $R$ and center $x$.
\begin{proof}
Replacing the radius $R$ by a smaller number does not affect these statements, as this causes the left hand sides to change by $O(a^{-m}e^{-\frac{R^2}{2a}})$.
Expanding the metric on $M$ in geodesic coordinates at $p_\ell$ as $\left.\<\cdot,\cdot\>\right\rvert_{x}=\left.\<\cdot,\cdot\>\right\rvert_{0}+\beta$ with $\beta=O(\|x\|^2)$ and $X_x=Ax$, one gets $\|X_x\|^2=\|Ax\|_0^2+\beta(Ax,Ax)=\|Ax\|_0^2+O(\|x\|^4)$. As $A$ is invertible, we can choose $R>0$ sufficiently small such that $c>0$ exists with $\left|\|X_x\|^2-\|Ax\|^2_0\right\rvert<c\|Ax\|^4_0$ on $\{x\in B_\ell\mid \|X_x\|<R\}$. Then replace $R$ by a smaller value such that $c\|Ax\|_0^4<\frac12\|Ax\|_0^2$ 
on $\{x\in B_\ell\mid \|X_x\|<R\}$. By the mean value Theorem $|y^{-s}-y_0^{-s}|\leq|y-y_0|\sup_{|t-y_0|<|y-y_0|}st^{-s-1}$ applied to $y\mapsto y^{-s}$, one obtains for $s\in\BR_0^+$, $x\neq0$
\begin{align}\label{estimate2}
\left|
\|X_x\|^{-2s}-\|Ax\|_0^{-2s}\right\rvert&\leq&c\|Ax\|_0^4\sup_{\frac12\|Ax\|_0^2<t<\frac32\|Ax\|_0^2}st^{-s-1}
= c2^{s+1}s(\|Ax\|_0^2)^{-s+1}.
\end{align}
The same way, for $s\geq0$
\begin{eqnarray}\label{estimate3}
\nonumber
\lefteqn{\left|
\|X_x\|^{-2s}e^{-\frac1{2a}\|X\|^2}-\|Ax\|_0^{-2s}e^{-\frac1{2a}\|Ax\|_0^2}\right\rvert}
\\ \nonumber&\leq&
c\|Ax\|_0^4\sup_{\frac12\|Ax\|_0^2<t<\frac32\|Ax\|_0^2}e^{-\frac{t}{2a}}(st^{-s-1}+\frac1{2a}t^{-s})
\\&=&ce^{-\frac1{4a}\|Ax\|_0^2}\left(2^{s+1}s(\|Ax\|_0^2)^{-s+1}+\frac{2^{s-1}}{a}(\|Ax\|_0^2)^{-s+2}\right).
\end{eqnarray}
Applying Proposition \ref{estimate1}(1) to the right hand side of equation (\ref{estimate3}) shows that one can replace $\|X\|$ by $\|Ax\|_0$ in the integrals $\int_{B_\ell}\frac{\tilde\eta\wedge\nu^{\pp-1}}{a^m(\frac12\|X\|^2)^{s}}e^{-\frac{1}{2a}\|X\|^{2}}$ for $m+s\leq\pp$, $0\leq s\leq\pp$ up to a term $O(a)$ as $a\to0^+$. 
Similarly for $B_\ell''(a):=\{x\in\BR^{2\pp}\mid \|Ax\|^2_0<2a\}$
\begin{eqnarray*}
\int_{\{x\mid \|Ax\|_0^2<2a\}}\frac{1}{(\|Ax\|_0^2)^{\pp-1}}\,d\lambda
&\stackrel{y=Ax}=&\nonumber
\frac1{|\det A|}\int_{B^{2\pp}_{\sqrt{2a}}(0)}\frac{1}{\|y\|_0^{2\pp-2}}\,d\lambda
\\&=&
\frac{\vol(S^{2\pp-1})}{|\det A|}\int_0^{\sqrt{2a}}r\,dr
=O(a).
\end{eqnarray*}
combined with (\ref{estimate2}) provides this replacement in
$\int_{B_\ell''(a)}\tilde\eta
\frac{\nu^{\pp-1}}{(\frac12\|X\|^2)^{\pp}}\left(
e^{-\frac1{2a}\|X\|^2}-1
\right)$ for the remaining factor-$(-1)$-term.

Furthermore the integration range $B_\ell'(a)$ can be replaced by $B_\ell''(a)$ for $m+s=\pp$, $0\leq s\leq\pp$: One has
$\{x\in B_\ell\mid \|Ax\|^2_0+c\|Ax\|^4_0<2a\}\subset B_\ell'(a)\subset\{x\in B_\ell\mid \|Ax\|^2_0-c\|Ax\|^4_0<2a\}$, where $\|Ax\|_0^2<\frac1{2c}$ on $B_\ell$. Thus one finds
$(B_\ell'(a)\setminus B_\ell''(a))\cup(B_\ell''(a)\setminus B_\ell'(a))\subset
\{x\in B_\ell\mid |\|Ax\|^2_0-2a|<c\|Ax\|^4_0\mbox{ and } \|Ax\|_0^2<\frac1{2c}\}$. The solutions $A_\pm=\|Ax\|_0$ of $\|Ax\|^2_0\pm c\|Ax\|^4_0=2a$ verify $|A_\pm-\sqrt{2a}|<c'a^{3/2}$ for $a$ sufficiently small.
Then integrals over the difference between $B_\ell'(a)$, $B_\ell''(a)$ are bounded by
\begin{eqnarray*}
&&\int_{\{x\mid |\|Ax\|^2_0-2a|<c\|Ax\|^4_0\}} a^{-m}(\|Ax\|^2_0)^{-s}e^{-\frac1{2a}\|Ax\|^2_0}\,d\lambda\\
&\leq&
\int_{\{x\mid |\|Ax\|^2_0-2a|<c\|Ax\|^4_0\}}a^{-m}(\|Ax\|^2_0)^{-s}\,d\lambda
\\&<&\frac{\vol (S^{2\pp-1})}{|\det A|}\left\{{a^{-m}\frac{r^{2\pp-2s}}{2\pp-2s}\Big|_{\sqrt{2a}-c'a^{3/2}}^{\sqrt{2a}+c'a^{3/2}}\atop\log r\Big|_{\sqrt{2a}-c'a^{3/2}}^{\sqrt{2a}+c'a^{3/2}}}\right.\mbox{ if }{s<\pp\atop s=\pp}
\\&=&O(a).
\end{eqnarray*}
Similarly replacing the integration range $B_\ell$ by $\{x\in\BR^{2\pp}\mid \|Ax\|^2_0<R^2\}$ causes a difference equal to $O(a^{-m}e^{-\frac{R^2}{3a}})$.

1) For $\pp=s+m$, $m\geq1$ and $a\to0^+$
\begin{eqnarray*}
\lefteqn{
\int_{\{x\mid \|Ax\|_0^2<R^2\}}\frac{1}{a^m\|Ax\|_0^{2s}}e^{-\frac{1}{2a}\|Ax\|_0^{2}}\,d\lambda
\stackrel{y=Ax}=
\frac1{|\det A|}\int_{B^{2\pp}_R(0)}\frac{1}{a^m\|y\|_0^{2s}}e^{-\frac{\|y\|_0^2}{2a}}\,d\lambda
}\\&=&
\frac{\vol(S^{2\pp-1})}{a^m|\det A|}\int_0^Rr^{2\pp-1-2s}e^{-\frac{r^2}{2a}}dr
\stackrel{u=\frac{r^2}{2a}}=
\frac{\vol(S^{2\pp-1})}{a^m|\det A|}\int_0^{\frac{R^2}{2a}}a\sqrt{2au}^{2\pp-2-2s}e^{-u}\,du
\\&=&
\frac{\vol(S^{2\pp-1})}{|\det A|}2^{m-1}(m-1)!+O(e^{-\frac{R^2}{2a}}).
\end{eqnarray*}

2) Let Ei denote the exponential integral function given by ${\rm Ei}(x)=-\int_{-x}^{+\infty}\frac{e^{-t}}{t}\,dt$ for $x\in\BR^-$. For $s=\pp$, $m=0$ one finds
\begin{eqnarray*}
\lefteqn{
\int_{\{x\mid \|Ax\|_0^2<2a\}}\frac{e^{-\frac{1}{2a}\|Ax\|_0^2}-1}{\|Ax\|_0^{2\pp}}\,d\lambda}
\\&\stackrel{y=Ax}=&
\frac1{|\det A|}\int_{B^{2\pp}_{\sqrt{2a}}(0)}\frac{e^{-\frac{\|y\|_0^2}{2a}}-1}{\|y\|_0^{2\pp}}\,d\lambda=
\frac{\vol(S^{2\pp-1})}{|\det A|}\int_0^{\sqrt{2a}}\frac{e^{-\frac{r^2}{2a}}-1}{r}\,dr
\\&\stackrel{u=\frac{r^2}{2a}}=&\frac{\vol(S^{2\pp-1})}{|\det A|}\int_0^1\frac{e^{-u}-1}{2u}\,du
=\frac{\vol(S^{2\pp-1})}{2|\det A|}(\Gamma'(1)+{\rm Ei}(-1)),
\end{eqnarray*}
\begin{eqnarray*}
\lefteqn{
\int_{\{x\mid 2a<\|Ax\|_0^2<R^2\}}\frac{e^{-\frac{1}{2a}\|Ax\|_0^2}}{\|Ax\|_0^{2\pp}}\,d\lambda}
\\&\stackrel{y=Ax}=&
\frac1{|\det A|}\int_{B^{2\pp}_R(0)\setminus B^{2\pp}_{\sqrt{2a}}(0)}\frac{e^{-\frac{\|y\|_0^2}{2a}}}{\|y\|_0^{2\pp}}\,d\lambda=
\frac{\vol(S^{2\pp-1})}{|\det A|}\int_{\sqrt{2a}}^R\frac{e^{-\frac{r^2}{2a}}}{r}\,dr
\\&\stackrel{u=\frac{r^2}{2a}}=&\frac{\vol(S^{2\pp-1})}{|\det A|}\int_1^{\frac{R^2}{2a}}\frac{e^{-u}}{2u}\,du
=\frac{\vol(S^{2\pp-1})}{2|\det A|}({\rm Ei}(-\frac{R^2}{2a})-{\rm Ei}(-1)).
\end{eqnarray*}
Using $|\det A|=\prod^{\pp}_{j=1}\theta_j^2$ the Proposition follows.
\end{proof}
Now one can verify as a refinement of equations (\ref{3.9BG}), (\ref{3.6BG})
\begin{lemma}\label{BGVerify}
For $t\to0^+$ and $\eta\in\frak A(M)$,
$\int_M\eta\cdot d_t=\int_{M_X}\eta (c_{\top,X}^{-1})'(TM)+O(t).$
\end{lemma}
\begin{proof}
One finds
\begin{align*}
\lefteqn{\int_M\eta\frac{\omega^{TM}}{2\pi t}\exp\left(\frac{d\XM^\flat}{4\pi i t}-\frac1{2t}\|\XM\|^2\right)
}\\
&\stackrel{{\rm Prop.\,}\ref{estimate1}(1)}=\int_{M}\eta\frac{\omega^{TM}}{2\pi t^{\pp}}\frac1{(\pp-1)!}\left(\frac{d\XM^\flat}{4\pi i}\right)^{\pp-1}e^{-\frac1{2t}\|\XM\|^2}+O(t)
\\&\stackrel{{\rm Prop.\,}\ref{limitsFora}(1)}=\sum_{\ell\in\J} a_\ell \frac{2^{\pp-1}\vol(S^{2\pp-1})}{\prod^{\pp}_{j=1}\theta_j^2}+O(t)
\\&\stackrel{{\rm Prop.\,}\ref{tildeEtaAtpl}\phantom{(1)}}=
\sum_{\ell\in\J} \eta_{p_\ell}(-i^{-\pp-1})\sum_j\frac1{\theta_j}\prod_j\frac1{\theta_j}+O(t)
\\&\stackrel{\phantom{{\rm Prop.\,}\ref{limitsFora}(1)}}=\int_{M_X}\eta_{p_\ell}(c_{\top,X}^{-1})'(TM)+O(t).
\end{align*}
\end{proof}
\begin{theor} Let $X$ have isolated zeros. For $a\to0^+$, $B_\ell'(a)=\{x\in B_\ell\mid \|X_x\|^2<2a\}$ and $\eta\in\frak A(M)$, 
\label{SFormel1}
\begin{eqnarray*}
\int_M\eta S_X(M,-\omega^{TM})&=&\int_M\eta\frac{\omega^{TM}}{2\pi}\left(\frac1{\frac12\|\XM\|^2-\frac{d\XM^\flat}{4\pi i}}\right)^{[<2\pp-2]}
\\&&+\int_{M\setminus\bigcup_{\ell\in\J} B_\ell'(a)}\eta\frac{\omega^{TM}}{2\pi}\left(\frac{d\XM^\flat}{4\pi i}\right)^{\pp-1}\left(\frac12\|X\|^2\right)^{-\pp}
\\&&+(\log a-2\Gamma'(1)-{\cal H}_{\pp-1})\int_{M_X}\eta (c_{\top,X}^{-1})'(TM)+O(a)
\end{eqnarray*}
where $\alpha^{[<2\pp-2]}$ denotes the part of degree less than $2\pp-2$ and ${\cal H}_{\pp-1}$ is the harmonic number as in eq. (\ref{harmonic}).
\end{theor}
\begin{proof}
On $M\setminus\bigcup_{\ell\in\J} B_\ell$, $\left|\int_{M\setminus\bigcup_{\ell\in\J} B_\ell}\frac{\tilde\eta}{\nu-\frac12\|X\|^2}e^{\frac{\nu-\frac12\|X\|^2}a}\right\rvert=o(e^{-\frac{C'}a})$ for a constant $C'>0$ depending on a lower bound for $\left.\|X\|^2\right\rvert_{M\setminus\bigcup_\ell B_\ell}$.
On $B_\ell$ one gets according to equation (\ref{limaSX}) (generalized to this situation by Lemma \ref{BGVerify}) and the previous Propositions
\begin{eqnarray*}
\lefteqn{
\int_M\eta S_X(M,-\omega^{TM})\stackrel{(5){,\rm Lemma}\, \ref{BGVerify}}=\int_M\eta \frac{\omega^{TM}}{2\pi}\cdot\frac{\exp(d_X\frac{\XM^\flat}{4\pi i a})-1}{d_X\frac{\XM^\flat}{4\pi i}}
}\\&&+(\log a-\Gamma'(1))\int_{M_X}\eta(c_{\top,X}^{-1})'(N_{M_X/M},g^{TM})+O(a)
\\&\stackrel{{\rm Prop.\,}\ref{estimate1}(2)}=&\int_{B_\ell}\eta\frac{\omega^{TM}}{2\pi}\cdot
\frac{\nu^{\pp-1}(1-e^{\frac{-\|X\|^2}{2a}})}{(\frac12\|X\|^2)^{\pp}}
\\&&+\int_{B_\ell}\eta\frac{\omega^{TM}}{2\pi}\left(
\Bigg(\frac1{\frac12\|X\|^2-\nu}\right)^{[<2\pp-2]}
\\&&-e^{\frac{-\|X\|^2}{2a}}
\cdot\nu^{\pp-1}\sum_{m=0}^{\pp-1}\frac{(\frac12\|X\|^2)^{-\pp+m}}{m!a^{m}}
\Bigg)
\\&&+(\log a-\Gamma'(1))\int_{M_X}\eta(c_{\top,X}^{-1})'(N_{M_X/M},g^{TM})+O(a)
\\&\stackrel{{\rm Prop.\,}\ref{limitsFora}}=&\int_{M\setminus\bigcup_{\ell\in\J} B_\ell'(a)}\eta\frac{\omega^{TM}}{2\pi}\cdot\frac{\nu^{\pp-1}}{\left(\frac12\|X\|^2\right)^{\pp}}
\\&&+\int_{B_\ell}\eta\frac{\omega^{TM}}{2\pi}
\left(\frac1{\frac12\|X\|^2-\nu}\right)^{[<2\pp-2]}
\\&&-\frac{\vol(S^{2\pp-1})}{\prod^{\pp}\theta_j^2}2^{\pp-1}
a_\ell\Big(\sum_{m=1}^{\pp-1}
\frac1m+\Gamma'(1)\Big)
\\&&+(\log a-\Gamma'(1))\int_{M_X}\eta(c_{\top,X}^{-1})'(N_{M_X/M},g^{TM})+O(a).
\end{eqnarray*}
Using the value of $a_\ell$ as given in Proposition \ref{tildeEtaAtpl} finishes the proof.
\end{proof}
We check Corollary \ref{ScalingSX} at this point:
\begin{cor}\label{ScalingSXeigene}
Assume that $X$ has isolated zeros. Then
$$
c^{N_\infty/2+1}S_{cX}(M,-\omega^{TM})-S_X(M,-\omega^{TM})=\log c^2\cdot(c_{\top,X}^{-1})'(N_{M_X/M},g^{TM})\cdot\delta_{M_X}
$$
in $P^M_{X,M_X}/P^{M,0}_{X,M_X}$.
\end{cor}
\begin{proof}
Consider $\eta_1\in\bigoplus_q\frak A^{q,q}(M)$ and set $\eta_t:=t^{-N_\infty/2}\eta=\sum_qt^{-q/2}\eta^{[q]}$ for $t\in\BR^+$. By Theorem \ref{SFormel1},
\begin{align*}
\lefteqn{\int_M\eta_1 \left(t^{\frac{N-2\pp}2}S_{tX}(M,-\omega^{TM})\right)=
\int_M\eta_t S_{tX}(M,-\omega^{TM})}
\\=&\sum_{j=1}^{\pp-1}\int_M\eta_t^{[2\pp-2j]}\frac{\omega^{TM}}{2\pi}\cdot\frac{t^{-j-1}\nu^{j-1}}{(\frac12\|\XM\|^2)^j}
\\&+t^{-\pp-1}\int_{M\setminus\bigcup_{\ell\in\J} B_\ell'(a/t^2)}\eta_1^{[0]}\frac{\omega^{TM}}{2\pi}\cdot\frac{\nu^{\pp-1}}{\left(\frac12\|X\|^2\right)^{\pp}}
\\&+(\log \frac{a}{t^2}+\log t^2-2\Gamma'(1)-{\cal H}_{\pp-1})t^{-\pp-1}\int_{M_X}\eta_1^{[0]} \cdot(c_{\top,X}^{-1})'(TM)+O(a)
\\=&t^{-\pp-1}\int_M\eta_1 S_{X}(M,-\omega^{TM})+\log(t^2)t^{-\pp-1}\int_{M_X}\eta_1^{[0]} \cdot(c_{\top,X}^{-1})'(TM).
\end{align*}
\end{proof}
When $d_X\eta_1=0$, then $d_{tX}\eta_t=0$.

\section{\texorpdfstring{The equivariant Bott-Chern current on the projective plane}{The equivariant Bott-Chern current on the projective plane}}
Consider the line bundle ${\cal L}:={\cal O}(1)$ on the projective line $M:={\bf
P}^1\BC$ with the chart
\begin{eqnarray*}
\psi:]0,2\pi[\times]-\pi/2,\pi/2[&\to&{\bf P}^1\BC\subset{\bf R}^3\\
(v,u)&\mapsto&\left(
\begin{array}{c}
\cos u\,\cos v\\ \cos u\,\sin v\\ \sin u
\end{array}\right).\\
\end{eqnarray*}
In these coordinates, the complex structure $J^{TM}$ is given by
$J^{TM}\frac\partial{\partial v}=\cos u\frac\partial{\partial u}$. As before, let $\Omega^{TM}$ denote the curvature of $TM$, which in this case equals the curvature tensor of $M$.
For any $\SO(3)$-invariant metric we find $-\Omega^{TM}(\frac\partial{\partial v},\frac\partial{\partial u})\frac\partial{\partial v}=\cos^2u\cdot\frac\partial{\partial u}$. Thus the 
Fubini-Study form $\omega^{TM}:=2\pi c_1({\cal O}(1))=\pi c_1(TM)=\frac12J^{TM}\Omega^{TM}$ is given by
$\omega^{TM}(\frac\partial{\partial v},\frac\partial{\partial u})=\frac{\cos u}2$.
Then $${\rm vol}({\bf P}^1\BC)=\int_{{\bf P}^1\BC}
\frac{\omega^{TM}}{2\pi}=1.$$ Consider the circle action induced by the vector
field
$X:=\frac\partial{\partial v}$. Thus, $\|X\|^2=\frac{\cos^2u}2$ and
$$
m^{TM}=\nabla^{TM}_\cdot X=\sin u\cdot J^{TM}.
$$
As $TM\cong {\cal O}(2)$ as ${\bf SU}(2)$-equivariant vector bundles, we find $m^{{\cal O}(1)}=\frac{i}2\sin u$ for the corresponding $\frak{su}(2)$-action.
\begin{Rem}
When instead considering the action of $\frak u(2)$, there is an additional non-trivial constant action of multiples of ${\rm id}_{\BR^2}$ on ${\cal O}(1)$, which induces a constant summand for $m^{{\cal O}(1)}$. We shall not do this in this paper.
\end{Rem}
For $\mu:=-im^{{\cal O}(1)}$ 
one finds
$$
d\mu=\frac12\cos u\,du=\iota_X(\omega^{TM})
$$
as in \cite[(2.4)]{BG}, except that the sign of $\omega^{TM}$ is chosen differently.

Let $\eta\in C^\infty({\bf P}^1\BC,\BC)$. Thus in the above coordinates $\frac1{2\pi}\int\left.\eta\right\rvert_{({u\atop v})}\,dv$ depends smoothly on $u$ and thus on $\sin u$. 
Assume that the moment map of $X$ on ${\cal O}(1)$ at the north pole is given by $m^{{\cal O}(1)}=\frac{i}2$. Hence it acts by $m^{N}=i\theta=i$ on $TM=N$ at this point. Thus for the normal bundle $N\to\{p\}$ at any fixed point $p$, one has $(c_{\top,tX}^{-1})'(\mtr N)
=\frac{-1}{c_{\top,tX}(\mtr N)}\sum_{\theta}\frac1{c_1(\bar N_\theta)+it\theta}=\frac{-1}{(it\theta)^2}=\frac{1}{t^2}.
$
\begin{theor}\label{SKasIntegral+}
For $\eta\in C^\infty({\bf P}^1\BC,\BC)$ set
$$
g(\sin u,v):=\left.\eta\right\rvert_{({u\atop v})}\quad\mbox{and}\quad\tilde g(r):=\frac1{2\pi}\int_0^{2\pi}\frac{g(r,v)+g(-r,v)}{2}\,dv.
$$
Then $\int_{M_X}\eta (c_{\top,tX}^{-1})'(TM)=\frac{2\tilde g(1)}{t^2}$ and
\begin{multline*}
\int_M\eta S_{tX}(M,-\omega^{TM})
\\=
\int_{-1}^{1}
\left(
\frac{2\tilde g(r)}{t^2}
-\frac{2\tilde g(1)}{t^2}
\right)\cdot\frac{dr}{1-r^2}
+(\log t^2-2\Gamma'(1))\cdot \frac{2\tilde g(1)}{t^2}.
\end{multline*}
\end{theor}
\begin{proof}
The equivariant Bott-Chern current $S$ is $X$-invariant. It switches sign under the isometry $r:({u\atop v})\to ({-u\atop v})$, as $X$ is $r$-invariant and $r^*\omega^{TM}=-\omega^{TM}$. As $r$ changes the orientation, $\int_M(r^*\eta)S_X(M,-\omega^{TM})=\int_M\eta S_X(M,-\omega^{TM})$. Hence in the integrals $\eta$ and $g$ can be replaced by their mean value $\tilde\eta$, $\tilde g$ over the compact orbit of these symmetries. Let $\tilde\eta_0:=\tilde g(1)$ denote the value of $\tilde\eta$ at the poles.
Applying Theorem \ref{SFormel1}
  results in
\begin{eqnarray*}
\lefteqn{\int_M\eta S_{tX}(M,-\omega^{TM})}
\\&=&\int_{M}(\tilde\eta-\tilde\eta_0)\frac{\omega^{TM}}{2\pi}\left(\frac12\|tX\|^2\right)^{-\pp}
+\int_{M\setminus\bigcup_\ell B_\ell'(a)}\tilde\eta_0\frac{\omega^{TM}}{2\pi}\left(\frac12\|tX\|^2\right)^{-\pp}
\\&&+(\log a-2\Gamma'(1)-{\cal H}_{\pp-1})\int_{M_X}\tilde\eta (c_{\top,tX}^{-1})'(TM)+O(a)
\\&=&\int_{-\pi/2}^{\pi/2}
\left(
\tilde g(\sin u)
-\tilde g(1)
\right)\cdot\frac{\cos u}{4\pi}(\frac{t^2\cos^2u}4)^{-1}\,2\pi \,du
\\&&+\int_{-\cos^{-1}\frac{2\sqrt{a}}t}^{\cos^{-1}\frac{2\sqrt{a}}t}
\tilde g(1)
\cdot\frac{\cos u}{4\pi}(\frac{t^2\cos^2u}4)^{-1}\,2\pi \,du
\\&&+(\log a-2\Gamma'(1))\cdot 2 \tilde g(1)\cdot t^{-2}+O(a).
\end{eqnarray*}
Using the substitution $r:=\sin u$ in both integrals on the right hand side, we get the expression in the Theorem as the second integral equals
\[
\tilde g(1)\cdot\int_{-\sqrt{1-4a/t^2}}^{\sqrt{1-4a/t^2}}\frac{2\,dr}{t^2(1-r^2)}=-\tilde g(1)\cdot\frac4{t^2}{\rm Artanh\,}
\sqrt{1-\frac{4a}{t^2}}=
\tilde g(1)\cdot \frac2{t^2}\log\frac{t^2}{a}+O(a).
\]
\end{proof}
The expression in Theorem \ref{SKasIntegral+} can be given a more combinatorial form for $\eta$ analytic. For $m\geq0$ let ${\cal H}_m=\sum_{j=1}^m\frac1j$ denote the harmonic numbers as in eq. (\ref{harmonic}), in particular ${\cal H}_0=0$.
\begin{theor}\label{SKwithStar+}
Set $\phi(r)^\#:=\sum_{m=1}^\infty \phi_m (2{\cal H}_{2m-1}-{\cal H}_{m-1})$ for any complex power series $\phi(r)=\sum_{m=0}^\infty \phi_mr^{2m}$. Assume that $\tilde g(r)$ is analytic at $r=0$ with radius of convergence $>1$. Then
\begin{eqnarray*}
\int_M\eta S_{tX}(M,-\omega^{TM})&=&
-\left(\frac{2\tilde g(r)}{t^2}\right)^\#
+(\log t^2-2\Gamma'(1))\cdot \frac{2\tilde g(1)}{t^2}.
\end{eqnarray*}
\end{theor}
\begin{proof}
The integrands Laurent expansion by $t$ in
$\int_{-1}^{1}
\left(
\frac{2\tilde g(r)}{t^2}
-\frac{2\tilde g(1)}{t^2}
\right)\cdot\frac{dr}{1-r^2}$
provides integrals of the form
$$
\int_{-1}^1(1-r^{2m})\frac{dr}{1-r^2}
=\sum_{j=1}^{m}\frac2{2j-1}=\left\{
{2{\cal H}_{2m-1}-{\cal H}_{m-1}\atop0}
\mbox{ if }{m>0\atop m=0}
\right.
$$
for $m\geq0$. The result follows by Theorem \ref{SKasIntegral+}.
\end{proof}
For the computation of the torsion form we shall need the following variant.
\begin{theor}\label{SKasIntegral}\label{SKwithStar}
Set $\phi(t)^*:=\sum_{m=0}^\infty \phi_m(2{\cal H}_{2m+1}-{\cal H}_{m})t^{2m}$ for any complex Laurent power series $\phi(t)=\sum_{m=-1}^\infty \phi_mt^{2m}$. For $\eta\in C^\infty({\bf P}^1\BC,\BC)$ define $g$,$\tilde g$ as in Theorem \ref{SKasIntegral+} and set $\eta_{t|({u\atop v})}:=\left.\eta\right\rvert_{({\arcsin(t\sin u)\atop v})}$ for $t\in[-1,1]$.
\begin{enumerate}
\item $\int_{M_X}\eta_t (c_{\top,tX}^{-1})'(TM)=\frac{2\tilde g(t)}{t^2}$ and
\begin{multline*}
\int_M\eta_t S_{tX}(M,-\omega^{TM})
\\=
\int_{-1}^{1}
\left(
\frac{2\tilde g(tr)}{t^2}
-\frac{2\tilde g(t)}{t^2}
\right)\cdot\frac{dr}{1-r^2}
+(\log t^2-2\Gamma'(1))\cdot \frac{2\tilde g(t)}{t^2}.
\end{multline*}
\item Assume that $\tilde g(t)$ is analytic at $t=0$ with radius of convergence $>1$. Then
\begin{multline*}
\int_M\eta_t S_{tX}(M,-\omega^{TM})=
-\left(\frac{2\tilde g(t)}{t^2}\right)^*
+(\log t^2-2\Gamma'(1))\cdot \frac{2\tilde g(t)}{t^2}
\\=-\left(\int_{M_X}\eta_t (c_{\top,tX}^{-1})'(TM)\right)^*+(\log t^2-2\Gamma'(1))\int_{M_X}\eta_t (c_{\top,tX}^{-1})'(TM).
\end{multline*}
\end{enumerate}
\end{theor}
\begin{proof}
This follows immediately from Theorems \ref{SKasIntegral+}, \ref{SKwithStar+} by replacing $\eta$ with $\eta_t$ and using $\left(\frac{\tilde g(t)}{t^2}\right)^*=\left(\frac{\tilde g(tr)}{t^2}\right)^\#$.
\end{proof}

\section{The defining property of $S_X$}
The equivariant Bott-Chern current verifies the critical relation
\begin{theor}\label{criticalBottChern}(\cite[Th. 3.9]{BG}) Using the inverse of the equivariant top Chern class, the identity
$$\frac{\bar\partial_X\partial_X}{2\pi i} S_X(M,-\omega^{TM})=
1-c_{\top,X}^{-1}(\mtr{N_{M_X/M}})\delta_{M_X}
$$
holds.
\end{theor}
In this section we shall quickly illustrate how this relation can be seen using the formula in Theorem \ref{SKasIntegral+}. Because of
$$
\left(
\frac{\bar\partial_{X}\partial_{X}}{2\pi i}\eta
\right)^{[0]}
=2\pi i\eta^{[2]}(X^{0,1},X^{1,0})-X^{1,0}.\eta^{[0]},
$$
when setting $\eta=:f_1\omega^{TM}+f_0$ with $f_0,f_1\in C^\infty(M)$, Theorem \ref{criticalBottChern} translates to
\begin{align}\label{DefSKP1-1}
2\pi i\int_Mf_1\cdot\omega^{TM}(X^{0,1},X^{1,0})S_X(M,-\omega^{TM})&=\int_Mf_1\cdot\omega^{TM},
\\
-\int_M(X^{1,0}.f_0)S_X(M,-\omega^{TM})&=-\int_{M_X}f_0 c_{\top,X}^{-1}(N_{M_X/M}).
\label{DefSKP1-2}
\end{align}
In the coordinates $u,v$ as above, $X^{1,0}=\frac12(\frac\partial{\partial v}-i\cos u\frac\partial{\partial u})$ and $2\pi i\omega^{TM}(X^{0,1},X^{1,0})=\frac\pi2\cos^2u$. Assume w.l.o.g. that $f_1$ and $X^{1,0}.f_0$ are invariant under $X$ and under $({u\atop v})\mapsto({-u\atop v})$. The $X$-invariance of $X^{1,0}.f_0$ is in fact equivalent to the $X$-invariance of $f_0$, as the equation $\frac{\partial^2}{\partial v^2} f_0=0$ for the real part implies $\frac\partial{\partial v} f_0=$const.. And $f_0$ is periodic in $v$, thus $\frac\partial{\partial v} f_0=0$.  
Now with
$$
\tilde g_1(\sin u):=\frac\pi2 f_1(\left({u\atop v}\right))\cos^2 u,\qquad \tilde g_1(\pm1):=0
$$
as in Theorem \ref{SKasIntegral+}, equation (\ref{DefSKP1-1}) is equivalent to
\begin{multline*}
2\pi i\int_Mf_1\cdot\omega^{TM}(X^{0,1},X^{1,0})S_X(M,-\omega^{TM})
\\=
\int_{-1}^12\tilde g_1(r)\frac{dr}{1-r^2}
=\int_0^{2\pi}\int_{-\pi/2}^{\pi/2}\frac12f_1(\left({u\atop v}\right))\cos u\,du\wedge dv
=
\int_Mf_1\cdot\omega^{TM}.
\end{multline*}
The real part of $X^{1,0}.f_0=\frac12\frac{\partial f_0}{\partial v}-\frac{i}2\cos u\cdot \frac{\partial f_0}{\partial u}$ does not contribute because of the $X$-invariance of $X^{1,0}.f_0$. Setting as in Theorem \ref{SKasIntegral+}
\begin{eqnarray*}
\tilde g_0(\sin u):=\frac{-i\cos u}2\cdot\frac{\partial f_0}{\partial u},\qquad \tilde g_0(\pm1):=0
\end{eqnarray*}
one finds
\begin{eqnarray*}
-\int_M(X^{1,0}.f_0)S_X(M,-\omega^{TM})&=&
\int_{-\pi/2}^{\pi/2}i\frac{\partial f_0}{\partial u}\,du
\\&=&if_0(\left({\pi/2\atop v}\right))-if_0(\left({-\pi/2\atop v}\right)).
\end{eqnarray*}
Using $\left.c_{\top,X}^{-1}(N_{M_X/M})\right\rvert_{N}=\frac1{i}=-i$, $\left.c_{\top,X}^{-1}(N_{M_X/M})\right\rvert_{S}=-\frac1{i}=i$, equation ({\ref{DefSKP1-2}}) and thus the equation in Theorem \ref{criticalBottChern} follows.

\section{\texorpdfstring{The height of ${\bf P}^1_\BZ$}{The height of P1Z}}\label{TheheightofP1Z}
One of the applications of Bismut's equivariant Bott-Chern current is a residue formula (in the spirit of Bott's formula) in Arakelov geometry (\cite{KR2}). In this section we verify that the formula gives the correct classically well-known value for the height of the projective plane over ${\rm Spec}\,\BZ$. We refer to \cite{Soule} for the concepts of Arakelov geometry and for the associated notations. By \cite[p. 70]{Soule}, the height of the projective plane $f:{\bf P}^1_\BZ\to{\rm Spec}\,\BZ$ with respect to the line bundle ${\cal L}:={\cal O}(1)$ is given by
$
\widehat{\deg}(f_*\hat c_1(\overline{{\cal O}(1)})^2)\in\BR
$ in terms of the Arakelov characteristic class $\hat c_1$ having values in the Gillet-Soul\'e intersection theory $\widehat{CH}({\bf P}^1_\BZ)$. Let ${\cal T}:={\rm Spec}\,\BZ[X,X^{-1}]$ be the one-dimensional torus group scheme and consider its canonical action on ${\bf P}^1_\BZ$ with fixed point scheme consisting of two copies of ${\rm Spec}\,\BZ$.
Let $r$
denote the additive characteristic class which is defined in \cite[p. 90]{KR2} as
\begin{eqnarray*}
\left.r_X(L)\right\rvert_{p}&:=&
-\sum_{j\geq 0}\frac{(-c_1(L))^j}{(i\phi)^{j+1}}\left(
-2\Gamma'(1)+2\log|\phi|-\sum_{k=1}^j\frac1{k}
\right)\in H^\bullet(M_X)\\
\end{eqnarray*}
for $L$ a line bundle acted upon by $X$ with an angle $\phi\in{\bf R}$ at 
$p\in M_{X}$. 
According to the residue formula in Arakelov geometry proven in \cite[Th. 2.11]{KR2}, the height can be computed using equivariant Arakelov characteristic classes $\hat c_{1,t}$, $\hat c_{\top,t}$ and the normal bundle $\bar N$ as
\begin{multline}
\label{XR2main}
\widehat{\deg}(f_*\hat c_1(\mtr{{\cal O}(1)})^2)=\widehat{\deg}\left(f^{\cal T}_*\frac{\hat c_1(\overline{{\cal O}(1)})^2)}{\hat c_{\top,t}(\bar N)}\right)
\\+\frac12\int_{{\bf P}^1\BC}c_{1,X}(\overline{{\cal O}(1)})^2S_X({\bf P}^1\BC,-\omega^{T{\bf P}^1\BC})
-\frac12\int_{{\bf P}_{\cal T}^1\BC}c_{1,X}({\cal O}(1))^2\frac{r_X(N)}{c_{\top,X}(N)}.
\end{multline}
Classically, at the fixed point subscheme $\ar
c_1(\mtr {\cal L})=\ar c_1(\mtr N)/2=0$, thus the arithmetic term on the right hand side of the residue formula (\ref{XR2main}) vanishes.  At a fixed point $p$ let $tX$ act by an angle $\phi$ on ${\cal O}(1)$ and by an angle $\theta$ on $N$. In our case the angles $\phi$
and $\theta$ at the fixed points are given by $\pm \frac t 2$ and $\pm t$,
respectively. As in \cite[p. 98]{KR2},
\begin{multline*}
-\frac12\int_{{\bf P}_{\cal T}^1\BC}c_{1,X}({\cal L})^{2}\frac{r_X(N)}{c_{\top,X}(N)}
=
-\frac1{2}\sum_{p\in {{\bf P}_{\cal T}^1\BC}} \frac{ c_{1,X}({\cal L})^{2}} {
c_{\top,X}(N)} r_X(N)
\\=
\sum_{p\in {{\bf P}_{\cal T}^1\BC}}
\frac{\phi_p^{2}}{\prod_\theta\theta}\sum_\theta\frac{-\Gamma'(1)+
\log|\theta|}{\theta}
=-\frac{\Gamma'(1)}2+\frac14\log t^2.
\end{multline*}
According to Theorem \ref{SKwithStar} we get with $\eta_t=(m^{\cal L}(tX))^2=-\frac{t^2}4\sin^2 u$ and $\tilde g(t)=-\frac14t^2$
\begin{eqnarray*}
\frac12\int_{{{\bf P}^1\BC}}\eta_t S_{tX}({\bf P}^1\BC,-\omega^{T{\bf P}^1\BC})=\frac12(1-\frac12\log t^2+\Gamma'(1)).
\end{eqnarray*}
Hence the residue formula in Arakelov geometry \cite[Th. 2.11]{KR2} states in this case
\begin{equation}
\ar\deg f_*\ar c_1(\mtr {\cal L})^2=\frac1 2
\end{equation}
which is the well-known classical value (\cite[p. 71]{Soule}).

\section{\texorpdfstring{Lie algebra equivariant torsion on ${\bf P}^1\BC$}{X-equivariant torsion on P1C}}
We employ the following special case of Bismut-Goette's main result:
\begin{theor} (\cite[Th. 0.1]{BG})\label{BGmain}
For $|t|$ sufficiently small,
\begin{eqnarray*}
\lefteqn{
T_{e^{tX}}({\bf P}^1\BC,\mtr{{\cal O}(\ell)})-T_{{\rm id},tX}({\bf P}^1\BC,\mtr{{\cal O}(\ell)})
}\\&=&\int_{{\bf P}^1\BC}\Td_{tX}(\mtr{T{\bf P}^1\BC})\ch_{tX}(\mtr{{\cal O}(\ell)})S_{tX}({\bf P}^1\BC,-\omega^{T{\bf P}^1\BC})
\\&&-\int_{({\bf P}^1\BC)_X}\Td_{e^{tX}}(T{\bf P}^1\BC)\ch_{e^{tX}}({\cal O}(\ell))I_{tX}(N_{{\bf P}^1\BC_X}).
\end{eqnarray*}
\end{theor}
With the angle $t\theta$ of the operation of $g=e^{tX}$ on $T_pM$ at the fixed point $p$, $t\neq0$, we get
$$
\left.\Td_{e^{tX}}(TM)\ch_{e^{tX}}(E)I_{tX}(TM)\right\rvert_{p}
=\frac{\Tr g^E}{\det(1-(g^{TM})^{-1})}\sum_{k\in\BZ\setminus\{0\}\atop\theta}\frac{\log(1+\frac{t\theta}{2\pi k})}{it\theta+2k\pi i}.
$$
By \cite[(20) (appendix)]{Bi8add}, for $0<|t\theta|<2\pi$ the last term $I_{tX}({\bf P}^1\BC)$ equals
$$
\sum_{k\in\BZ\setminus\{0\}}\frac{\log(1+\frac{t\theta}{2\pi k})}{it\theta+2k\pi i}
=\sum_{m\geq1\atop m\rm\,odd}{\cal H}_m\frac{\zeta(-m)(it\theta)^m}{m!}
$$
using the harmonic numbers as given in eq. (\ref{harmonic}). For $M={\bf P}^1\BC$, $T{\bf P}^1\BC={\cal O}(2)$, $\theta^{{\cal O}(1)}=\pm1/2$ at the fixed points we get for the second summand on the right hand side in Theorem \ref{BGmain}
\begin{multline}\label{IclassP1C}
\int_{M_{X}}\Td_{e^{tX}}(TM)\ch_{e^{tX}}({\cal O}(\ell))I_{tX}(TM)
\\=\sum_p\frac{e^{\pm it\ell/2}}{1-e^{\mp it}}\sum_{m\geq1\atop m\rm\,odd}{\cal H}_m\frac{\zeta(-m)(\pm it)^m}{m!}
=\frac{\cos\frac{(\ell+1)t}2}{i\sin\frac{t}2}\sum_{m\geq1\atop m\rm\,odd}{\cal H}_m\frac{\zeta(-m)(it)^m}{m!}
.
\end{multline}
\begin{prop}\label{TdchS_P1}
For $M={\bf P}^1\BC$, $t\neq0$, the first summand on the right hand side in Theorem \ref{BGmain} is given by
\begin{eqnarray*}\lefteqn{
\int_M\Td_{tX}(\mtr{TM})\ch_{tX}(\mtr{{\cal O}(\ell)}) S_{tX}(M,-\omega^{TM})}
\\&=&
\int_{-1}^{1}
\left(
\frac{ r \cos\frac{(\ell+1)tr}2}{\sin\frac{tr}2}
-\frac{\cos\frac{(\ell+1)t}2}{\sin\frac{t}2}
\right)\cdot\frac{dr}{t(1-r^2)}
+(\log t^2-2\Gamma'(1))\cdot \frac{\cos\frac{(\ell+1)t}2}{t\sin\frac{t}2}
\\&=&-\left(
\frac{\cos\frac{(\ell+1)t}2)}{t\sin\frac{t}2}
\right)^*+(\log t^2-2\Gamma'(1))\cdot \frac{\cos\frac{(\ell+1)t}2}{t\sin\frac{t}2}.
\end{eqnarray*}
\end{prop}
\begin{proof}
Setting $\mua:=m^{{\cal O}(1)}$ the $X$-equivariant classes are given by
\begin{eqnarray*}
\eta:=\Td_{tX}(TM)\ch_{tX}({\cal O}(\ell))
&=&
\frac{2t\mua}{1-e^{-2t\mua}}e^{t\ell\mua}+\mbox{terms of higher degree}
.
\end{eqnarray*}
Thus we get
$$\left.\Td_{tX}(TM)\ch_{tX}({\cal O}(\ell))(c_{\top,tX}^{-1})'(\mtr N)\right\rvert_{p}
=\frac{1}{t^2}\frac{\pm it e^{\pm it\ell/2}}{1-e^{\mp it}}$$
and henceforth
$$
\int_{M_{X}}\Td_{tX}(TM)\ch_{tX}({\cal O}(\ell))
(c_{\top,tX}^{-1})'(\mtr N)
=\frac{\cos\frac{(\ell+1)t}2}{t\sin \frac{t}2}
.
$$
Remember that by its definition via the integral $\int_M\eta\wedge d_t$, where $d_t$ as in eq. (\ref{dt}) is a form of degree 2 and higher, in this complex-1-dimensional case $\int\eta S_{tX}(M,-\omega^{TM})$ only depends on $\eta^{[0]}$ for any form $\eta$.
Hence the result follows by applying Theorem \ref{SKasIntegral} with $\tilde g(t)=\frac{t \cos(\frac{(\ell+1)t}2)}{2\sin(\frac{t}2)}$.
\end{proof}
\begin{theor}\label{g-eq.Torsion}
For $0<t<2\pi$, the value in Proposition \ref{TdchS_P1} has for $\ell\to+\infty$ an asymptotic expansion given by
\begin{multline*}
\int_M\Td_{tX}(\mtr{TM})\ch_{tX}(\mtr{{\cal O}(\ell)}) S_{tX}(M,-\omega^{TM})
\\=
\frac{
-\cos\frac{(\ell+1)t}2}
{t\sin\frac{t}2}\log(\ell+1)
+
\frac{\sin\frac{(\ell+1)t}2\cdot \frac\pi2
-\cos\frac{(\ell+1)t}2\cdot\left(\Gamma'(1)-\log t
\right)}
{t\sin\frac{t}2}
+O(\frac1\ell).
\end{multline*}
\end{theor}
Note that $\log(\ell+1)=\log\ell+O(\frac1\ell)$.
\begin{proof}We decompose the integral in Proposition \ref{TdchS_P1} as
\begin{multline*}
\int_{-1}^{1}
\left(
\frac{ r \cos\frac{(\ell+1)tr}2}{\sin\frac{tr}2}
-\frac{\cos\frac{(\ell+1)t}2}{\sin\frac{t}2}
\right)\cdot\frac{dr}{t(1-r^2)}
\\=
\int_{-1}^{1}
\frac{\cos\frac{(\ell+1)tr}2-\cos\frac{(\ell+1)t}2}
{(1-r^2)t\sin\frac{t}2}\,dr
+
\int_{-1}^{1}
\left(
\frac{r}{\sin\frac{tr}2}-\frac{1}{\sin\frac{t}2}
\right)\cdot\frac{\cos\frac{(\ell+1)tr}2}{t(1-r^2)}\,dr.
\end{multline*}
For $|t|<2\pi$, the factor $f(r):=\left(
\frac{r}{\sin\frac{tr}2}-\frac{1}{\sin\frac{t}2}
\right)\cdot\frac{1}{t(1-r^2)}$ in the second integral on the right hand side is smooth on $r\in[-1,1]$. Partial integration shows
\begin{multline*}
\int_{-1}^{1}
\left(
\frac{r}{\sin\frac{tr}2}-\frac{1}{\sin\frac{t}2}
\right)\cdot\frac{\cos\frac{(\ell+1)tr}2}{t(1-r^2)}\,dr
\\=\frac2{(\ell+1)t}\sin\frac{(\ell+1)tr}2\cdot f(r)\Big|_{-1}^1
-\frac2{(\ell+1)t}\int_{-1}^1\sin\frac{(\ell+1)tr}2\cdot f'(r)\,dr
=O(1/\ell).
\end{multline*}
To estimate the first integral on the right hand side, we represent it as twice the integral over $[0,1]$, decompose $\frac1{1-r^2}=\frac{1/2}{1-r}+\frac{1/2}{1+r}$ and use the trigonometric addition formula for $\cos\frac{(\ell+1)tr}2=\cos\left(\frac{(\ell+1)t(r-1)}2+\frac{(\ell+1)t}2\right)$.  Let ${\rm Si}$ and ${\rm Ci}$ denote the sine and cosine integral functions, respectively, which are given by ${\rm Si}(x)=\int_0^x\frac{\sin t}{t}\,dt$ and ${\rm Ci}(x)=-\int_x^{+\infty}\frac{\cos t}{t}\,dt$ for $x\in\BR^+$. Then the above integral equals
\begin{eqnarray*}
\lefteqn{
\int_{-1}^{1}
\frac{\cos\frac{(\ell+1)tr}2-\cos\frac{(\ell+1)t}2}
{(1-r^2)t\sin\frac{t}2}\,dr
}\\&=&
\frac{\sin\frac{(\ell+1)t}2\cdot {\rm Si}((\ell+1)t)
-\cos\frac{(\ell+1)t}2\cdot\left(-\Gamma'(1)-{\rm Ci}((\ell+1)t)+\log((\ell+1)t)
\right)}
{t\sin\frac{t}2}
\\&=&\frac{\sin\frac{(\ell+1)t}2\cdot \frac\pi2
-\cos\frac{(\ell+1)t}2\cdot\left(-\Gamma'(1)+\log((\ell+1)t)
\right)}
{t\sin\frac{t}2}+O(\frac1\ell).
\end{eqnarray*}
Adding the term $(\log t^2-2\Gamma'(1))\cdot \frac{\cos\frac{(\ell+1)t}2}{t\sin\frac{t}2}$ from Proposition \ref{TdchS_P1}, one obtains the result.
\end{proof}
By iterating the partial integration, one can extend this expansion to arbitrary negative powers of $\ell+1$.

\begin{theor}\label{gEquivTorsion}
With respect to the action of the vector field $X\in\Gamma({\bf P}^1\BC,T{\bf P}^1\BC)$, the $X$-equivariant torsion is given by
\begin{eqnarray*}
T_{{\rm id},tX}({\bf P}^1\BC,\mtr{{\cal O}(\ell)})&=&
-\frac{\cos\frac{(\ell+1)t}2}{\sin\frac{t}2}\sum_{m\geq1\atop m{\rm\,odd}}\left(2\zeta'(-m)
+{\cal H}_m\zeta(-m)\right)\frac{(-1)^{\frac{m+1}2}t^m}{m!}
\\&&+\sum_{m=1}^{|\ell+1|}\frac{\sin(2m-|\ell+1|)\frac{t}2}{\sin\frac{t}2}\log m
+\left(
\frac{\cos\frac{(\ell+1)t}2)}{t\sin\frac{t}2}
\right)^*
\end{eqnarray*}
where $(t^{2m})^*:=t^{2m}\cdot\left\{
{2{\cal H}_{2m+1}-{\cal H}_{m}\atop0}
\mbox{ if }{m\geq0\atop m=-1}
\right.$ (as in Theorem \ref{SKwithStar}) and ${\cal H}_{m}$ is the harmonic number as in eq. (\ref{harmonic}).
\end{theor}
The first summand contains exactly the function defining the (non-equivariant) Gillet-Soul\'e $R$-class \cite[p. 160]{Soule},
\begin{equation}
R({\cal L})=\sum_{m\geq1\atop m{\rm\,odd}}\left(2\zeta'(-m)
+{\cal H}_m\zeta(-m)\right)\frac{c_1({\cal L})^m}{m!}
\end{equation}
(see Th. \ref{CompareARR1} for a closer analysis). The $\zeta'$-term as well as the equivariant-metric-terms are derived from the equivariant torsion. The $*$-summand originates from the equivariant Bott-Chern current and the ${\cal H}_m\zeta(-m)$-term is the $I$-class. Some terms from the first two summands cancel each other.
\begin{proof}
\cite[Th. 2]{K1} shows for $t\in]0,2\pi[$,
\begin{eqnarray}\label{eqTorsionP1}
T_{e^{tX}}({\bf P}^1\BC,\mtr{{\cal O}(\ell)})&=&
2R^{\rm rot}(t)\frac{\cos\frac{(\ell+1)t}2}{\sin\frac{t}2}
+\sum_{m=1}^{|\ell+1|}\frac{\sin(2m-|\ell+1|)\frac{t}2}{\sin\frac{t}2}\log m
\end{eqnarray}
where according to \cite[Prop. 1]{K1},
\begin{eqnarray}\label{RrotDef}
R^{\rm rot}(t)=\frac{-\Gamma'(1)
+\log t}{t}-\sum_{m\geq1\atop m{\rm\,odd}}\zeta'(-m)(-1)^{\frac{m+1}2}\frac{t^m}{m!}.
\end{eqnarray}
Combining this with Bismut-Goette's Theorem \ref{BGmain}, Proposition \ref{TdchS_P1} and equation \ref{TdchS_P1} we find
\begin{align*}
T_{{\rm id},tX}({\bf P}^1\BC,\mtr{{\cal O}(\ell)})=&
2\left(
\frac{-\Gamma'(1)
+\log t}{t}-\sum_{m\geq1\atop m{\rm\,odd}}\zeta'(-m)(-1)^{\frac{m+1}2}\frac{t^m}{m!}
\right)\frac{\cos\frac{(\ell+1)t}2}{\sin\frac{t}2}
\\&+\sum_{m=1}^{|\ell+1|}\frac{\sin(2m-|\ell+1|)\frac{t}2}{\sin\frac{t}2}\log m
-(\log t^2-2\Gamma'(1))\cdot \frac{\cos\frac{(\ell+1)t}2}{t\sin\frac{t}2}
\\&+\left(
\frac{\cos\frac{(\ell+1)t}2)}{t\sin\frac{t}2}
\right)^*
+
\frac{\cos\frac{(\ell+1)t}2}{i\sin\frac{t}2}\sum_{m\geq1\atop m\rm\,odd}{\cal H}_m\frac{\zeta(-m)(it)^m}{m!}.
\end{align*}
\end{proof}

This sum does not contain a factor $\log t$ nor any negative powers of $t$ anymore. It is an even power series in $t$. 
The expansion in $t$ up to $O(t^4)$ is given by
\begin{eqnarray}\nonumber
\lefteqn{T_{{\rm id},tX}({\bf P}^1\BC,\mtr{{\cal O}(\ell)})}
\\&=&\nonumber
4\zeta'(-1)
+\sum_{m=1}^{|\ell+1|}(2m-|\ell+1|)\log m
-\frac{|\ell+1|^2}2
\\&&+\Bigg(\frac{10|1+\ell |^4-5|1+\ell |^2-4}{720}
\nonumber+\frac{-4 \zeta'(-3) - (|1+\ell|^2-1|) \zeta'(-1)}6
\\&&+
\sum_{m=1}^{|\ell+1}\frac{ (|\ell+1| - 2 m)^3 -(|\ell+1| - 2 m)}{24}\log m
\Bigg)\cdot t^2+O(t^4).
\end{eqnarray}

\begin{Rem}\rm\label{vgl.bekannteTorsion}
We shall verify that the value of $T_{{\rm id},tX}({\bf P}^1\BC,\mtr{{\cal O}(\ell)})$ for $t=0$ equals the known formula for $T({\bf P}^1\BC,\mtr{{\cal O}(\ell)})$:

The equivariant torsion has been computed in \cite[Theorem 18]{K2} for equivariant vector bundles on symmetric spaces. We shall use the notations $\bz,\bzs,\chi^*$ etc. from \cite[p. 102]{K2}:
For $\phi\in{\bf R}$ and ${\rm Re}\,s>1$, consider the Lerch zeta function
\begin{equation}\label{DefZetaL}
\zeta_L(s,\phi)=\sum_{k=1}^\infty \frac{e^{ik\phi}}{k^s}.
\end{equation}
For $\vp$ fixed, the function $\zeta_L$ has analytic continuation in the variable $s$ to $\BC\setminus\{1\}$. Set $\zeta'_L(s,\phi):=\partial/\partial s
(\zeta_L(s,\phi))$. Let
$P:{\bf Z}\to\BC$ be a function of the form
\begin{equation*} P(k)=\sum_{j=0}^m c_jk^{n_j} e^{ik\phi_j}\label{sternnn}
\end{equation*} with $m\in{\bf N}_0$, $n_j\in{\bf N}_0$, $c_j\in\BC$,
$\phi_j\in{\bf R}$ for all $j$. Then for $p\in{\bf R}$ we shall use the notations $P^\odd(k):=(P(k)-P(-k))/2$,
\begin{eqnarray*}
\bz P&:=\sum_{j=0}^m c_j\zeta_L(-n_j,\phi_j),\qquad\qquad\qquad\quad
\bzs P&:=\sum_{j=0}^m c_j\zeta_L'(-n_j,\phi_j),\\
\bzo P&:=\sum_{j=0}^m c_j\zeta_L(-n_j,\phi_j)\sum_{\ell=1}^{n_j}\frac1
\ell,\quad
\mbox{Res }P(p)&:=\sum_{j=0\atop\phi_j\equiv 0{\rm\ mod }2\pi}^m c_j
\frac{p^{n_j+1}}{2(n_j+1)}
\end{eqnarray*}
and 
$$P^*(p):=-\sum_{j=0\atop\phi_j\equiv 0{\rm\ mod\,
}2\pi}^m c_j \frac{p^{n_j+1}}{4(n_j+1)} \sum_{\ell=1}^{n_j}\frac 1 \ell.$$

For ${\bf P}^1\BC={\bf U}(2)/{\bf U}(1)\x{\bf U}(1)=:G/K$ identify the Lie algebra of the maximal torus with $\BR^2$ with the ordering $(e_1,e_2)$. Then the positive roots are given by $\Delta^+=\psi=\{e_1-e_2\}$, and the weight providing ${\cal O}(\ell)$ is given by $\Lambda=-\ell\cdot e_2$. Thus for $\alpha=e_1-e_2$, $\rho_G=\frac\alpha2$ we get the dimension $\chi_{\rho_G+\Lambda+k\alpha}(0)=\frac{\<\alpha,\rho_G+\ell\cdot e_1+k\alpha\>}{\<\alpha,\rho_G\>}=1+\ell+2k$. Furthermore $(\alpha,\rho_G+\Lambda)=\frac{2\<\alpha,\rho_G+\ell\cdot e_1\>}{\<\alpha,\alpha\>}=1+\ell.$ The formula in \cite[Theorem 18]{K2} requires the highest weight $\Lambda$ of the bundle to be in the closure of the positive Weyl chamber. Thus for $\ell\geq0$, one obtains for the evaluation of the characters at the neutral element
\begin{eqnarray} \nonumber
T({\bf P}^1\BC,\mtr{{\cal O}(\ell)})&=&2\bzs
\sum_{\Psi}\chi_{\rho_G+\Lambda+k\a}^\odd
-2\sum_{\Psi}\chi_{\rho_G+\Lambda-k\a}^*\left((\a,\rho_G+\Lambda)\right)\nonumber\\
&&\nonumber{}-\sum_\Psi\sum_{k=1}^{(\a,\rho_G+\Lambda)}\chi_{\rho_G+\Lambda-k\a}\log
k-\sum_{\Psi} 
\bz\chi_{\rho_G+\Lambda+k\a}\log \frac{\|\a\|^2_\diamond}2\nonumber
\\&=&\nonumber4\zeta'(-1)-\frac{(\alpha,\rho_G+\Lambda)^2}2{\cal H}_1
-\sum_{k=1}^{\ell+1}(1+\ell-2k)\log k
\\&=&4\zeta'(-1)-\frac{(\ell+1)^2}2
-\sum_{k=1}^{\ell+1}(1+\ell-2k)\log k
\end{eqnarray}
for the choice $\frac{\|\a\|^2_\diamond}2=1$ as in \cite[(71)]{K2}.
\cite[Theorem 18]{K2} contained a mistyped sign in the third summand. A formula valid for arbitrary equivariant bundles (and without the typo) was given in \cite[Th. 5.2]{KK}, written slightly differently using $\bz P(k)=-\bz P(-k)-P(0)$. This result shows for any $\ell\in\BZ$ that
$
T({\bf P}^1\BC,\mtr{{\cal O}(\ell)})=4\zeta'(-1)-\frac{(\ell+1)^2}2
-\sum_{k=1}^{|\ell+1|}(|1+\ell|-2k)\log k$.
For arbitrary $X_0$ one gets an additional summand
\begin{align}\nonumber
-\bz\chi_{\rho+\ell\lambda+k\alpha}\log\frac{\|\alpha\|^2_\diamond}2&=-\bz(2k+\ell+1)\log\frac{\|\alpha\|^2_\diamond}2
\\&=-(2\cdot\frac{-1}{12}-\frac{\ell+1}2)\log\frac{\|\alpha\|^2_\diamond}2=(\frac23+\frac\ell2)\log\frac{\|\alpha\|^2_\diamond}2.
\end{align}
where $\frac2{\|\alpha\|_\diamond}=\frac1{\<\alpha,\rho_G\>_\diamond}=\vol_\diamond{\bf P}^1\BC$ by \cite[Cor. 7.27]{BeGeVe}. See also \cite[p. 840]{KMMW} for additional remarks.
\end{Rem}

\section{The torsion form}\label{FinalProof}
Let $P\to B$ be a ${\bf U}(2)$ principal bundle, ${\bf P}\to B$ the induced ${\bf P}^1\BC$-bundle and $E:=P\x_{{\bf U}(2)}\BC^2$. Then ${\bf P}={\bf P}(E)$. The curvature form $\Theta\in\Lambda^{1,1}T^*B\otimes (P\x_{{\bf U}(2)}{\frak u}(2))$ 
inserted in the torsion form as a $\BC$-valued homogeneous polynomial on $\frak u(2)$ provides an expression in terms of $c_1(E), c_2(E)$ via the fiber bundle embedding $P\x_{{\bf U}(2)}{\frak u}(2)\hookrightarrow P\x_{{\bf U}(2)}\End(\BC^2)=\End(E)$. In general \cite[(2.74)]{BG} shows for such bundles induced by principal bundles with compact structure group that the torsion form is a cohomology class.

\begin{Rem}\rm In general ${\bf P}^1\BC$-bundles can look more complicated; the relevant structure group is ${\bf PU}(2)={\bf SO}(3)$ and the obstruction is an element of $H^3(B,\BZ)$: The structure group of projective bundles is ${\bf PU}(k)={\bf U}(k)/{\bf U}(1)={\bf SU}(k)/(\BZ/k\BZ)$ (embedded diagonally). Thus one gets an obstruction $\alpha\in H^3(B,\BZ)$ with $k\alpha=0$ (\cite[p. 517]{Ivancevic}; \cite{AtA}; \cite{At}). See also \cite{Kot} for a more detailed discussion of the holomorphic situation.
\end{Rem}

\begin{proof} {\it (of Theorem \ref{mainTorsionformResult})}
Each $Y\in{\frak u}(2)$ induces a vector field on ${\bf P}^1\BC={\bf U}(2)/{\bf U}(1)\x{\bf U}(1)$, which we shall denote by $\rho(Y)$.
The Lie algebra element $\left({i/2\atop0}{0\atop -i/2}\right)\in\frak u(2)$ acts with period $\pi$ on ${\bf P}^1\BC$ and induces the vector field $X=\frac\partial{\partial v}$.
Thus the element $Y_0=\left({i\alpha\atop0}{0\atop i\beta}\right)\in\frak u(2)$ induces the vector field $(\alpha-\beta)\cdot X$. For any $Y\in\frak g$, $\gamma\in G$ the equality $T_{{\rm id},\rho({\Ad_\gamma Y})}({\bf P}^1\BC,\mtr{{\cal O}(\ell)})=T_{{\rm id},\rho(Y)}({\bf P}^1\BC,\mtr{{\cal O}(\ell)})$ holds, as both vector bundle and metric are $\gamma$-invariant. Thus $T_{{\rm id},tX}({\bf P}^1\BC,\mtr{{\cal O}(\ell)})$ determines $T_{{\rm id},\rho(Y)}({\bf P}^1\BC,\mtr{{\cal O}(\ell)})$ completely. Because of the $\Ad$-invariance of $\Tr,\det$ and $(\Tr Y_0)^2-4\det Y_0=-(\alpha-\beta)^2$,
\begin{equation}\label{TorsOhneFaktor}
T_{{\rm id},\rho(Y)}({\bf P}^1\BC,\mtr{{\cal O}(\ell)})=T_{iX\sqrt{(\Tr Y)^2-4\det Y}}({\bf P}^1\BC,\mtr{{\cal O}(\ell)}).
\end{equation}
Considering the determinant of the Euler sequence for the map $\pi:{\bf P}^1\BC\to$point
$$
0\to{\cal O}(-1)\to\pi^*\BC^2\to T{\bf P}^1\BC\otimes{\cal O}(-1)\to0
$$
one finds ${\cal O}(-2)\cong\pi^*\Lambda^2\BC^2\otimes T^*{\bf P}^1\BC$. As $e^{Y_0}\in{\bf U}(2)$ acts with weight $e^{i(\alpha+\beta)}$ on the pointwise trivial line bundle $\pi^*\Lambda^2\BC^2$, the action of $e^{Y_0}$ on ${\cal O}(\ell)$ is given by the action of the traceless component in $\frak{su}(2)$ composed with the pointwise factor $e^{-i\ell\frac{\alpha+\beta}2}$. Thus, when considering the torsion with respect to the action of the Lie algebra element $Y\in\frak g$ instead of the action of the vector field $\rho(Y)$, one gets the value (\ref{TorsOhneFaktor}) multiplied by $e^{-\frac\ell2\Tr Y}$.

According to \cite[(2.74)]{BG}, one obtains the torsion form $T_\pi(\mtr{{\cal O}(\ell)})$ by replacing $Y\in\frak g$ with $-\frac1{2\pi i}\Omega^E$. Thus in the value of $T_{{\rm id},tX}({\bf P}^1\BC,\mtr{{\cal O}(\ell)})$ given by Theorem \ref{gEquivTorsion}, $-t^2$ has to be replaced by $c_1(E)^2-4c_2(E)$, and the factor $e^{-\frac\ell2\Tr Y}$ gets replaced by $e^{-\frac\ell2c_1(E)}$.
\end{proof}
One can also verify quickly that the class $c_1(E)^2-4c_2(E)$ is invariant under $E\mapsto E\otimes{\cal L}'$ for every line bundle ${\cal L}'$. This verifies that it is indeed well-defined for ${\bf P}^1\BC$-bundles.
\section{Comparison with the arithmetic Grothendieck-Riemann-Roch Theorem}
Given a ${\bf P}^1\BC$-bundle $\pi:{\bf P}\to B$, we denote the vertical tangent space by $T\pi$.
\begin{prop}\label{liftCharClasses}
For any ${\bf P}^1\BC$-bundle $\pi:{\bf P}\to B$ one obtains
$$
\pi^*c_1(E)=c_1(T\pi)-2c_1({\cal O}(1))\quad\mbox{and}\quad\pi^*(c_1(E)^2-4c_2(E))=c_1(T\pi)^2.
$$
\end{prop}
Thus $\pi^*T_\pi(\mtr{{\cal O}(\ell)})=T_\ell(c_1(T\pi)^2)$.
\begin{proof}
Using the Euler sequence for projective fibrations
$$
0\to{\cal O}(-1)\to\pi^*E\to T\pi\otimes{\cal O}(-1)\to0
$$
one finds
\begin{multline*}
\pi^*\ch(E)=\ch({\cal O}(-1))(1+\ch(T\pi))
2+[2c_1({\cal O}(-1))+c_1(T\pi)]
\\+[c_1({\cal O}(-1))^2+c_1(T\pi)c_1({\cal O}(-1))+\frac12c_1(T\pi)^2]+\dots
\end{multline*}
and thus $\pi^*c_1(E)=c_1(T\pi)-2c_1({\cal O}(1))$ and
\[
c_1(T\pi)^2=\pi^*(-(\ch(E)^{[2]})^2+4\ch(E)^{[4]})=\pi^*(c_1(E)^2-4c_2(E)).
\]
\end{proof}
For a general ${\bf P}^1\BC$-bundle one has to replace $E$ by $H^0({\bf P}^1\BC,{\cal O}(1))$ in the result.
The arithmetic Grothendieck-Riemann-Roch Theorem (\cite{GRS}) states with the fibres $Z\cong{\bf P}^1\BC$ of the fibration ${\bf P}E\to B$
\begin{eqnarray*}
\widehat{\ch}(\pi_*\mtr{{\cal O}(\ell)})-T_\pi(\mtr{{\cal O}(\ell)})
=\pi_*(\widehat{\ch}(\mtr{{\cal O}(\ell)})\widehat{\Td}(\mtr{T\pi}))
-\int_{Z}\ch({\cal O}(\ell))\Td(T\pi)R(T\pi).
\end{eqnarray*}
\begin{theor}\label{CompareARR1}
When multiplied by  $e^{-\frac\ell2c_1(E)}$, the summand
$$
\tilde T_\ell(-t^2):=-\frac{\cos\frac{(\ell+1)t}2}{\sin\frac{t}2}\sum_{m\geq1\atop m{\rm\,odd}}\left(2\zeta'(-m)
+{\cal H}_m\zeta(-m)\right)\frac{(-1)^{\frac{m+1}2}t^m}{m!}
$$
of $T_\ell(-t^2)$ contributes the term $\int_{Z}\ch({\cal O}(\ell))\Td(T\pi)R(T\pi)$ to the torsion form.
\end{theor}
\begin{proof}
By the projection formula $\pi_*(\pi^*\alpha\wedge\beta)=\alpha\wedge\pi_*\beta$ in cohomology and Proposition \ref{liftCharClasses},
\begin{eqnarray}\label{VanishEven}
\int_{Z}c_1(T\pi)^{2m}
=(c_1(E)^2-4c_2(E))^{m}\cdot \int_{Z}0
=0.
\end{eqnarray}
Similarly, one finds
\begin{align}\label{ProjectionP1}\nonumber
\lefteqn{\int_{Z}c_1(T\pi)^{2m+1}
=\int_{Z}(2c_1({\cal O}(1))+\pi^*c_1(E))\cdot c_1(T\pi)^{2m}
}\\
=&2\int_{Z}c_1({\cal O}(1))\cdot\pi^*(c_1(E)^2-4c_2(E))^{m}
+\int_Z\pi^*(c_1(E)\cdot(c_1(E)^2-4c_2(E))^{m})
\nonumber
\\=&2(c_1(E)^2-4c_2(E))^{m}\cdot \int_{Z}c_1({\cal O}(1))
=2(c_1(E)^2-4c_2(E))^m.
\end{align}
Noticing $\tilde T_\ell(-(it)^2)=\frac{\cosh \frac{(\ell+1)t}2}{\sinh \frac{t}2}R(t)$, one gets
\begin{eqnarray*}
\lefteqn{\int_{Z}\ch({\cal O}(\ell))\Td(T\pi)R(T\pi)
=\int_Ze^{\ell c_1({\cal O}(1))}\frac{c_1(T\pi)}{1-e^{-c_1(T\pi)}}R(T\pi)
}\\&\stackrel{{\rm Prop.}\,\ref{liftCharClasses}}=&e^{-\frac\ell2c_1(E)}\int_Ze^{\frac\ell2 c_1(T\pi)}\frac{c_1(T\pi)}{1-e^{-c_1(T\pi)}}R(T\pi)
\\&\stackrel{(\ref{VanishEven})}=&e^{-\frac\ell2c_1(E)}\int_Z\frac{c_1(T\pi)}2\cdot\frac{\cosh \frac{(\ell+1)c_1(T\pi)}2}{\sinh \frac{c_1(T\pi)}2}R(T\pi)
\\&=&e^{-\frac\ell2c_1(E)}\int_Z\frac{c_1(T\pi)}2\cdot \tilde T_\ell(c_1(T\pi)^2)
\\&\stackrel{(\ref{ProjectionP1})}=&e^{-\frac\ell2c_1(E)}\tilde T_\ell(c_1(E)^2-4c_2(E)).
\end{eqnarray*}
\end{proof}

\section{Equivariant torsion forms}
Consider a holomorphic isometric action of a Lie group $G$ on $M$. Consider $g\in G$ and a vector field $X$ induced by an element of the Lie algebra $\frak z_{G}(g)\subset\frak g$ of the centralizer of $g$.
Let $I_{g,X}$ denote the additive equivariant characteristic class on $M_g\cap M_X$ which is given 
for a
line bundle $L$ as follows: If $X$ acts at the fixed point $p$ by an angle
$\theta'\in{\bf R}$ on
$L$ and $g$ acts by $e^{i\theta}$ with $\theta\in[0,2\pi[$, then for $|\theta'|$ sufficiently small
\begin{equation}
\left.I_{g,X}(L)\right\rvert_{p}:=\sum_{k\in\BZ\atop2\pi k+\theta\neq 0}\frac{\log(1+\frac{\theta'}{2\pi 
k+\theta})}{c_1(L)+i\theta+i\theta'+2 k\pi
i}.
\end{equation}
Set $\tilde R_0(\theta,x):=\sum_{k=0}^\infty\left(\frac\partial{\partial s}\zeta_L(-k,\theta)+\zeta_L(-k,\theta)\frac{{\cal H}_k}{2}\right)\frac{x^k}{k!}$ with $\zeta_L$ as in eq. (\ref{DefZetaL}). 
For $\theta\neq0$, set $R(\theta,x):=\tilde R_0(e^{i\theta},x)-\tilde R_0(e^{-i\theta},-x)$. In \cite[Prop. 1]{K1}, it is shown that $2iR^{\rm rot}(\theta)=R(\theta,0)$ for $\theta\in]0,2\pi[$ with $R^{\rm rot}$ given by eq. (\ref{RrotDef}).
Bismut-Goette showed in equation (\cite[(0.13)]{BG}) that for $\theta\neq0$, $|\theta'|$ sufficiently small,
\begin{eqnarray*}
\left.I_{g,X}(L)\right\rvert_{p}=R(\theta,c_1(L)+i\theta')-R(\theta+\theta',c_1(L)).
\end{eqnarray*}
We refer to \cite[Th. 2.7]{BG} for the definition of $(g,tX)$-equivariant characteristic classes $\Td_{g,tX}$, $\ch_{g,tX}$ and torsion $T_{g,tX}$.
Bismut-Goette's main result shows for the action of $g\in\SU(2)$ and the infinitesimal action of $X\in\frak{su}(2)$ on ${\bf P}^1\BC$:
\begin{theor}\cite[Th. 1]{BG}
Let $g\in\SU(2)$, $X\in\frak z_{\SU(2)}(g)$ act of ${\bf P}^1\BC$. Then the $(g,tX)$-equivariant torsion $T_{g,tX}({\bf P}^1\BC,\mtr{{\cal O}(\ell)})$ verifies for $|t|$ sufficiently small
\begin{eqnarray*}
\lefteqn{T_{ge^{tX}}({\bf P}^1\BC,\mtr{{\cal O}(\ell)})-T_{g,tX}({\bf P}^1\BC,\mtr{{\cal O}(\ell)})
}\\
&=&\int_{{\bf P}^1\BC_g}\Td_{g,tX}(\mtr{T{\bf P}^1\BC})\ch_{g,tX}(\mtr{{\cal O}(\ell)})S_{tX}({\bf P}^1\BC_g,-\omega^{T{\bf P}^1\BC})
\\&&-\int_{{\bf P}^1\BC_X\cap{\bf P}^1\BC_g}\Td_{ge^{tX}}(T{\bf P}^1\BC)\ch_{ge^{tX}}({\cal O}(\ell))I_{g,tX}(N_{{\bf P}^1\BC_g}).
\end{eqnarray*}
\end{theor}
If $g=:e^{sX}$ acts with isolated fixed points, the $S$-current term disappears, as the $S$-current has no degree $0$ part. Thus
\begin{multline*}
T_{e^{sX},tX}({\bf P}^1\BC,\mtr{{\cal O}(\ell)})=T_{e^{(s+t)X}}({\bf P}^1\BC,\mtr{{\cal O}(\ell)})
\\+\int_{{\bf P}^1\BC_X}\Td_{e^{(s+t)X}}(T{\bf P}^1\BC)\ch_{e^{(s+t)X}}({\cal O}(\ell))I_{e^{sX},tX}(T{\bf P}^1\BC_X).
\end{multline*}
Similarly to equation (\ref{IclassP1C}), one finds
\begin{multline*}
\sum_p\left.\Td_{e^{(s+t)X}}(TM)\ch_{e^{(s+t)X}}({\cal O}(\ell))I_{e^{sX},tX}(TM)\right\rvert_{p}
\\=-\frac{\cos\frac{(\ell+1)(s+t)}2}{\sin\frac{(s+t)}2}
\sum_{k\in\BZ}\frac{\log(1+\frac{t}{2\pi k+s})}{2\pi k+t+s}
.
\end{multline*}
Thus using equation (\ref{eqTorsionP1}) and equation (\cite[(0.13)]{BG}) one obtains
\begin{theor} \label{Gg-eq.Torsion}
The $(e^{sX},tX)$-equivariant torsion verifies for $|s|,|t|$ sufficiently small
\begin{eqnarray*}
\lefteqn{T_{e^{sX},tX}({\bf P}^1\BC,\mtr{{\cal O}(\ell)})
}\\&=&
2R^{\rm rot}(s+t)\frac{\cos\frac{(\ell+1)(s+t)}2}{\sin\frac{(s+t)}2}
+\sum_{m=1}^{|\ell+1|}\frac{\sin(2m-|\ell+1|)\frac{(s+t)}2}{\sin\frac{(s+t)}2}\log m
\\&&-\frac{\cos\frac{(\ell+1)(s+t)}2}{\sin\frac{(s+t)}2}
\sum_{k\in\BZ}\frac{\log(1+\frac{t}{2\pi k+s})}{2\pi k+t+s}
\\&=&\sum_{m=1}^{|\ell+1|}\frac{\sin(2m-|\ell+1|)\frac{(s+t)}2}{\sin\frac{(s+t)}2}\log m
+\frac{\cos\frac{(\ell+1)(s+t)}2}{i\sin\frac{(s+t)}2}R(t,is).
\end{eqnarray*}
\end{theor}
This provides the value of the $G$-equivariant torsion form introduced in \cite{Ma} analogous to Theorem \ref{mainTorsionformResult}.


\begin{thebibliography}{100}

\bibitem[At]{At}Atiyah, M.:
K-theory past and present. Sitzungsberichte der Berliner Mathematischen Gesellschaft, 411--417, Berliner Math. Ges., Berlin, 2001. 

\bibitem[AtA]{AtA}Atiyah, M. F., Anderson, D. W.: K-theory. Mathematics Lecture Notes {\bf 7}. W. A. Benjamin (New York and Amsterdam), 1967.

\bibitem[AS4]{AS4}Atiyah, M. F., Singer, I. M.: The index of elliptic operators. IV. Ann. of Math. {\bf 93} (1971), 119--138.

\bibitem[BeGeVe]{BeGeVe}Berline, N., Getzler, E., Vergne, M.: Heat kernels and Dirac operators. Grundlehren der Mathematischen Wissenschaften, {\bf 298}. Springer-Verlag, Berlin, 1992.

\bibitem[B1]{Bi1}Bismut, J.-M.: The Atiyah-Singer index theorem for families of Dirac operators: two heat equation proofs. 
Invent. Math. {\bf 83}, no. 1 (1986), 91--151. 

\bibitem[B2]{Bi8}Bismut, J.-M.:
Equivariant Bott-Chern currents and the Ray-Singer analytic torsion. 
Math. Ann. {\bf 287}, no. 3 (1990), 495--507. 

\bibitem[B3]{Bi8add}Bismut, J.-M.:
Koszul complexes, harmonic oscillators, and the Todd class. With an appendix by the author and C. Soul\'e. J. Amer. Math. Soc. {\bf 3} (1990), 159--256.

\bibitem[B4]{Bi9}Bismut, J.-M.: Complex equivariant intersection, excess normal bundles and Bott-Chern currents. Comm. Math. Phys. {\bf 148}, no. 1 (1992), 1--55.

\bibitem[B5]{Bi10}Bismut, J.-M.: 
Hypoelliptic Laplacian and Bott-Chern cohomology. 
A theorem of Riemann-Roch-Grothendieck in complex geometry. Prog. Math. {\bf 305}. Birkh\"auser/Springer, Cham, 2013.

\bibitem[B6]{Bi11R}Bismut, J.-M.: Holomorphic families of immersions and higher analytic torsion forms. Ast\'erisque no. {\bf 244} (1997), viii+275 pp.

\bibitem[B7]{Bi12R}Bismut, J.-M.: Quillen metrics and singular fibres in arbitrary relative dimension. J. Algebraic Geom. {\bf 6}, no. 1 (1997), 19-149.

\bibitem[B8]{Bi13R}Bismut, J.-M.: Holomorphic and de Rham torsion. Compos. Math. {\bf 140}, no. 5 (2004), 1302-1356.

\bibitem[BGS2]{BGS2}Bismut, J.-M., Gillet, H., Soulé, C.: Analytic torsion and holomorphic determinant bundles. II. Direct images and Bott-Chern forms. Comm. Math. Phys. {\bf 115}, no. 1 (1988), 79--126. 

\bibitem[BG]{BG}Bismut, J.-M.; Goette, S.: Holomorphic equivariant analytic torsions. Geom. Funct. Anal. {\bf 10}, no. 6 (2000), 1289--1422.

\bibitem[BK]{BK}Bismut, J.-M., K\"ohler, K.: Higher analytic torsion forms for direct images and anomaly formulas, J. Alg. Geom. {\bf 1} (1992), 647--684.

\bibitem[BuFrLi]{BuFrLi}Burgos Gil, J. I., Freixas i Montplet, G., Liţcanu, R.: Generalized holomorphic analytic torsion. J. Eur. Math. Soc. {\bf 16}, no. 3 (2014), 463--535.

\bibitem[CoKot]{Kot}Coelho, R., Kotschick, D.: Lifts of projective bundles and applications to string manifolds. Bull. London Math. Soc. {\bf 53} (2021) 470--481.

\bibitem[Fa]{Fa}Faltings, G.: Lectures on the arithmetic Riemann-Roch theorem. Notes taken by Shouwu Zhang. Ann. Math. Studies, {\bf 127}. Princeton University Press, Princeton, NJ, 1992.

\bibitem[Fi]{Fi}Finski, S.: On the full asymptotics of analytic torsion. J. Funct. Anal. {\bf 275} , no. 12 (2018), 3457-3503.

\bibitem[FZ]{FZha}Fu, L., Zhang, Y.: Motivic integration and the birational invariance of the BCOV invariant. Sel. Math. New Ser. {\bf 29}, 25 (2023).

\bibitem[GS1]{GS1}Gillet, H., Soul\'e, C.: Direct images of Hermitian holomorphic bundles. Bull. Amer. Math. Soc. {\bf 15}, no. 2 (1986), 209--212.

\bibitem[GS3]{GS}Gillet, H., Soul\'e, C.: {Characteristic classes for algebraic vector bundles with hermitian metric. I}. Ann. Math. {\bf 131} (1990), 163--203.

\bibitem[GSZ]{GSZ}Gillet, H., Soul\'e, C.: Analytic torsion and the arithmetic Todd genus. With an appendix by D. Zagier. Topology {\bf 30}, no. 1 (1991), 21--54.

\bibitem[GRS]{GRS}Gillet, H., R\"ossler, D., Soul\'e, C.: An arithmetic Riemann-Roch theorem in higher degrees. Ann. Inst. Fourier {\bf 58}, no. 6 (2008), 2169--2189.

\bibitem[GS4]{GS4}Gillet, H., Soul\'e, C.: An arithmetic Riemann-Roch theorem. Invent. Math. 110, no. 3, (1992), 473--543.

\bibitem[GrH]{GH} Griffiths, P., Harris, J.: Principles of Algebraic Geometry. Wiley Interscience 1978.

\bibitem[Iv]{Ivancevic}Ivancevic, V. G., Ivancevic, T. T.: Applied differential geometry. A modern introduction. World Scientific Publishing Co. Pte. Ltd., Hackensack, NJ, 2007.

\bibitem[KaK]{KK}Kaiser, C., K\"ohler, K.: A fixed point formula of Lefschetz type in Arakelov geometry. III. Representations of Chevalley schemes and heights of flag varieties. Invent. Math. {\bf 147}, no. 3 (2002), 633--669.

\bibitem[KlMaMarW]{KMMW}Klevtsov, S., Ma, X., Marinescu, G., Wiegmann, P.:
Quantum hall effect and Quillen metric. Comm. Math. Phys. {\bf 349}, no. 3 (2017), 819--855.

\bibitem[K1]{K1}K\"ohler, K.:
Equivariant analytic torsion on ${\bf P}^n\BC$. 
Math. Ann. {\bf 297}, no. 3 (1993), 553--565.

\bibitem[K2]{K2}K\"ohler, K.:
Holomorphic torsion on Hermitian symmetric spaces, J. reine angew. Math. {\bf 460} (1995), 93--116.

\bibitem[K3]{K3}K\"ohler, K.: Complex analytic torsion forms for torus fibrations and moduli spaces, 166--195 in: Regulators in Analysis, Geometry and Number Theory, N. Schappacher, A. Reznikov (ed.), Progress in Math. {\bf 171}, Birkh\"auser 2000

\bibitem[KR2]{KR2}K\"ohler, K., Roessler, D.: A fixed point formula of Lefschetz type in Arakelov geometry. II. A residue formula. Ann. Inst. Fourier (Grenoble) {\bf 52}, no. 1 (2002), 81--103.

\bibitem[MQ]{MQ}Mathai, V., Quillen, D.:
Superconnections, Thom classes, and equivariant differential forms. 
Topology {\bf 25}, no. 1 (1986), 85--110. 

\bibitem[Ma]{Ma}Ma, X.:
Submersions and equivariant Quillen metrics. 
Ann. Inst. Fourier {\bf 50}, no. 5 (2000), 1539--1588.

\bibitem[Mou]{Mourougane}Mourougane, C.: Analytic torsion of Hirzebruch surfaces. Math. Ann. {\bf 335}, no. 1 (2006), 221--247.

\bibitem[P]{Puchol} Puchol, M.: The asymptotic of the holomorphic analytic torsion forms, J. Lond. Math. Soc. {\bf 108} (2023),80--140.

\bibitem[S]{Soule} Soul\'e, C., Abramovich, D., Burnol, J. F., Kramer, J.: Lectures on Arakelov Geometry (Cambridge studies in math. {\bf 33}): Cambridge University Press 1992.

\bibitem[Zha]{Zha}Zha, Y.: A general arithmetic Riemann-Roch theorem. Ph.D. thesis, Univ. of Chicago (1998).

\bibitem[Z]{Zha2}Zhang, Y.: BCOV invariant and blow-up. Compos. Math. {\bf 159}, no. 4 (2023), 780 -- 829.

\end{thebibliography}
\end{document}